\documentclass[11pt,a4paper]{article}
\date{}
\makeatletter
\@addtoreset{equation}{section}
\makeatother
\usepackage[colorlinks,
            linkcolor=blue,
            anchorcolor=blue,
            citecolor=blue]{hyperref}
\usepackage{graphicx}
\usepackage{hyperref}
\hypersetup{hypertex=true,
            colorlinks=true,
            linkcolor=blue,
            anchorcolor=blue,
            citecolor=blue}
\usepackage{hyperref,amsmath,amssymb,amscd}
\usepackage{amssymb,amsmath,enumerate}
\usepackage{indentfirst}
\usepackage{url}
\usepackage[capitalise]{cleveref}
\usepackage{amsthm}
\usepackage{cite}
\newtheorem{theorem}{Theorem}[section]

\newtheorem{lemma}[theorem]{Lemma}
\newtheorem{proposition}[theorem]{Proposition}
\theoremstyle{definition}

\newtheorem{remark}[theorem]{\bf Remark}
\allowdisplaybreaks[4]

\begin{document}
\title{\bf
{Existence and asymptotical behavior of normalized solutions to focusing biharmonic HLS upper critical Hartree equation with a local perturbation}}
\author{{Jianlun Liu$^1$\footnote{E-mail address: jianlunliumath@163.com (J. Liu)},\ \ Hong-Rui Sun$^1$\footnote{E-mail address: hrsun@lzu.edu.cn (H.R. Sun)}\ \ \mbox{and}\ \ Ziheng Zhang$^2$\footnote{Corresponding author. E-mail address: zhzh@mail.bnu.edu.cn (Z. Zhang)}}\\
{\small \emph{$^1$School of Mathematics and Statistics,\  Lanzhou University,\ Lanzhou {\rm730000},}}\\
{\small \emph{People's Republic of China}}\\
{\small \emph{$^2$School of Mathematical Sciences,\  Tiangong University,\ Tianjin {\rm300387},}}\\
{\small \emph{People's Republic of China}}\\
}
\maketitle
\baselineskip 17pt

{\bf Abstract}: This paper is concerned with the following focusing biharmonic HLS upper critical Hartree equation with a local perturbation
$$
\begin{cases}
	{\Delta}^2u-\lambda u-\mu|u|^{p-2}u-(I_\alpha*|u|^{4^*_\alpha})|u|^{4^*_\alpha-2}u=0\ \ \mbox{in}\ \mathbb{R}^N, \\[0.1cm]
	\int_{\mathbb{R}^N} u^2 dx = c,
\end{cases}		
$$
where $0<\alpha<N$, $N \geq 5$, $\mu,c>0$, $2+\frac{8}{N}=:\bar{p}\leq p<4^*:=\frac{2N}{N-4}$, $4^*_\alpha:=\frac{N+\alpha}{N-4}$, $\lambda \in \mathbb{R}$ is a Lagrange multiplier and $I_\alpha$ is the Riesz potential. Choosing an appropriate testing function, one can derive some reasonable estimate on the mountain pass level. Based on this point, we show the existence of normalized solutions by verifying the \emph{(PS)} condition at the corresponding mountain pass level for any $\mu>0$.  The contribution of this paper is that the recent results obtained for $L^2$-subcritical perturbation by Chen et al. (J. Geom. Anal. 33, 371 (2023)) is extended to the case $\bar{p}\leq p<4^*$. Moreover, we also discuss asymptotic behavior of the energy to the mountain pass solution when $\mu\to 0^+$ and $c\to 0^+$, respectively.



\section{Introduction and main results}
In this paper, we investigate the following focusing biharmonic Hartree equation with a local perturbation
\begin{equation}\label{eqn:BS-equation-L2-Super+Critical}
\begin{cases}
	{\Delta}^2u-\lambda u-\mu |u|^{p-2}u-(I_\alpha*|u|^{4^*_\alpha})|u|^{4^*_\alpha-2}u=0  \ \ \mbox{in}\ \mathbb{R}^N, \\[0.1cm]
	\int_{\mathbb{R}^N} u^2 dx = c,  \\[0.1cm]
\end{cases}
\end{equation}
where $0<\alpha<N$, $N\geq 5$, $\mu,c>0$, $*$ is the convolution product on $\mathbb R^N$, $\lambda$ appears as a Lagrange multiplier, $2+\frac{8}{N}=:\bar{p}\leq p<4^*:=\frac{2N}{N-4}$ and $4^*_\alpha:=\frac{N+\alpha}{N-4}$. In the sense of Hardy-Littlewood-Sobolev (abbreviated as HLS) inequality, we refer to $4^*_\alpha$ as the upper critical exponent. Meanwhile, for every $x\in\mathbb R^N\setminus\{0\}$, the Riesz potential $I_\alpha$ is of the following expression
\begin{equation}\label{eqn:Defn-I-A-alpha}
I_\alpha(x):=\frac{A_\alpha}{|x|^{N-\alpha}},
\end{equation}
where $A_\alpha:=\frac{\Gamma(\frac{N-\alpha}{2})}{\Gamma(\frac{\alpha}{2})\pi^\frac{N}{2}2^\alpha}$ and $\Gamma$ is the Gamma function which can be found in
\cite[Page:19]{Riesz1949}.

As we know, the interest in studying problem (\ref{eqn:BS-equation-L2-Super+Critical}) comes from searching for standing waves with the form $\psi(t,x)=e^{-i\lambda t}u(x)$ to the following time-dependent biharmonic nonlinear elliptic equation
\begin{equation}\label{eqn:intial-equation}
-i\partial_{t}\psi-\gamma\Delta^2 \psi+\mu |\psi|^{p-2}\psi+\kappa(I_\alpha*|\psi|^{q})|\psi|^{q-2}\psi=0\ \ \mbox{in}\ \mathbb{R}\times\mathbb{R}^N,
\end{equation}
where $\gamma,\mu>0$, $\kappa=\pm1$, $t$ denotes the time, $\psi: \mathbb R^N \times \mathbb R\to \mathbb C$ is a complex valued function, $i$ denotes the imaginary unit, $\Delta^2$ is the biharmonic operator was proposed in \cite{Karpman1996,Karpman-Shagalov2000}, $\frac{N+\alpha}{N}\leq q \leq 4^*_\alpha$ and $2<p\leq 4^*$. From the physical standpoint, the real number $\kappa$ refers to the defocusing versus focusing regime, that is, $\kappa=-1$ corresponds to the defocusing problem, while $\kappa=1$ stands for the focusing problem. For more details, see \cite{Fibich-Ilan-Papanicolaou2002,Fibich-Ilan-Schochet2003,Gao-Wang2014,Kenig-Merle2008,Pausader2007} and the references mentioned there. By recalling (\ref{eqn:intial-equation}), we see that the nonlinear term $\kappa(I_\alpha*|\psi|^{q})|\psi|^{q-2}\psi$ is called the Hartree type nonlinearity, which is relevant to describing several physical phenomena: the dynamics of the mean-field limits of many-body quantum systems such as coherent states and condensates, the quantum transport in semiconductors superlattices, the study of mesoscopic structures in Chemistry \cite{Elgart-Schlein2007,Lieb-Yau1987,Yu-Gaididei-Rasmussen-Christiansen1995}. For more backgrounds, we refer the reader to  \cite{Gao2013,Miao-Xu-Zhao2008,Yonggeun-Gyeongha-Soonsik-Sanghyuk2015} and the references listed therein. In recent years, some researchers begun to pay their attention to the biharmonic Hartree equation (that is, $\mu=0$ in (\ref{eqn:intial-equation})). Explicitly, Banquet et al. \cite{Banquet-Villamizar-Roa2020} proved some local and global well-posedness results in $H^s$ for the Cauchy problem associated to the biharmonic Hartree equation with variable dispersion coefficients. Later, Saanouni \cite{Saanouni2020} studied a dichotomy of scattering of global solutions versus finite time blow-up in the inter-critical regime and Cao et al. explored the stationary case in \cite{Cao-Dai2019}.

If $\mu\neq0$ and $\kappa=0$, equation (\ref{eqn:intial-equation}) reduces to the following so-called biharmonic Schr\"{o}dinger equation
 \begin{equation}\label{eqn:intial-equation-BiNLS}
-i\partial_{t}\psi-\gamma\Delta^2 \psi+\mu |\psi|^{p-2}\psi=0\ \ \mbox{in}\ \mathbb{R}\times\mathbb{R}^N,
\end{equation}
which was considered in \cite{Ivanov-Kosevich1983,Turitsyn1985} to study the stability of solitons in magnetic materials once the effective quasi particle mass becomes infinite. For some recent studies on (\ref{eqn:intial-equation-BiNLS}), one can see \cite{Miao-Xu-Zhao2009,Pausader2009,Zhang-Zheng2010} and the references listed within. During the process of research, many people focus on finding standing waves solutions to (\ref{eqn:intial-equation-BiNLS}) with a prescribed mass, since the mass may have specific meanings, such as the total number of atoms in Bose-Einstein condensation, or
the power supply in nonlinear optics. These solutions are commonly called normalized solutions, which provide valuable insights into dynamical properties of stationary solutions, such as the stability or instability of orbits. For this situation, the interested authors can refer to \cite{Bellazzini-Visciglia2010,Phan2018,Ma-Chang-2022}.
If $\mu\neq0$ and $\kappa=1$ in (\ref{eqn:intial-equation}), we obtain a biharmonic Hartree equation with one local perturbation. To the best our of knowledge, based on the conservation of mass theory, Chen and Chen in \cite{Chen-chen2023} paid their attention to standing waves solutions of (\ref{eqn:intial-equation}) when $2<p<\bar{p}:=2+\frac{8}{N}$ (called $L^2$-subcritical exponent), that is, solutions of problem (\ref{eqn:BS-equation-L2-Super+Critical}).
To obtain this type (weak) solutions for problem (\ref{eqn:BS-equation-L2-Super+Critical}), it is equivalent to search for critical points of the $C^1$ functional $I:H^2(\mathbb R^N)\to\mathbb R$
\begin{equation}\label{eqn:Defn-Iu}
I(u)=\frac{1}{2}\int_{\mathbb{R}^N}|\Delta u|^2 dx-\frac{\mu}{p}\int_{\mathbb{R}^N}|u|^p dx-\frac{1}{24^*_\alpha}\int_{\mathbb{R}^N}(I_\alpha*|u|^{4^*_\alpha})|u|^{4^*_\alpha} dx,\ \forall u\in H^2({\mathbb{R}^N})
\end{equation}
restricted to the constraint
\begin{equation}\label{eqn:Sc}
S(c):=\Bigl\{u\in H^2(\mathbb{R}^N):\int_{\mathbb{R}^N}|u|^2dx=c\Bigr\}.
\end{equation}
However, it is clear that $I$ defined in (\ref{eqn:Defn-Iu}) is unbounded from below on $S(c)$. To overcome this difficulty, the following Poho\v{z}aev manifold
$$
\mathcal P(c):=\{u\in S(c):P(u)=0\}
$$
is necessary, where
\begin{equation}\label{eqn:Pohozaev-identify}
P(u):=\int_{\mathbb{R}^N}|\Delta u|^2 dx-\mu\gamma_p\int_{\mathbb{R}^N}|u|^p dx-\int_{\mathbb{R}^N}(I_\alpha*|u|^{4^*_\alpha})|u|^{4^*_\alpha} dx.
\end{equation}
Noting that any critical point of $I|_{S(c)}$ stays in $\mathcal P(u)$ (see \cite[Page 4]{Chen-chen2023}) and the fact that this Poho\v{z}aev manifold is a natural constraint, the authors in \cite{Chen-chen2023} proved that problem (\ref{eqn:BS-equation-L2-Super+Critical}) with $2<p<\bar{p}$ possesses a normalized ground state solution and a mountain pass type solution, see its Theorems 1.1 and 1.2. Moreover, they also discussed asymptotic behavior of the mountain pass type solution as $\mu\to 0^+$ and $c\to 0^+$ in their Theorem 1.3. Obviously, the results presented in \cite{Chen-chen2023} are not involved in the $L^2$-critical perturbation or the $L^2$-supercritical perturbation, that is, $p=\bar{p}$ or $\bar{p}<p<4^*$ in problem (\ref{eqn:BS-equation-L2-Super+Critical}). Therefore, a natural issue is to investigate the existence of solutions for problem (\ref{eqn:BS-equation-L2-Super+Critical}) when $\bar{p}\leq p<4^*$.

At this moment, we can state our main result as follows.
\begin{theorem}\label{Thm:normalized-bScritical-solutions-L2super+L2critical}
%For any $\mu,c>0$ and $\bar{p}\leq p<4^*$. In particular, we futher assume that $c\in(0,c_0)$ when $p=\bar{p}$.
Suppose that
$$
\begin{cases}
\mu,c>0,\ \frac{22}{3}<p<10, & \mbox{if}\,\,  N=5;\\
\mu,c>0, \ \bar{p}<p<4^*, & \mbox{if}\,\,  N\geq 6; \\
\mu>0, c\in(0,c_0), \ \bar{p}=p, & \mbox{if}\,\,  N\geq 6, \\
\end{cases}
$$
then problem (\ref{eqn:BS-equation-L2-Super+Critical}) possesses a nonnegative solution $u_{\mu,c}$ satisfies $0<I(u_{\mu,c})<\frac{4+\alpha}{2(N+\alpha)}S^\frac{N+\alpha}{4+\alpha}_\alpha$, where $S_\alpha$ and $c_0$ are given in (\ref{eqn:S-alpha-define}) and (\ref{eqn:defn-c0}), respectively. Meanwhile, the corresponding Lagrange multiplier $\lambda_{\mu,c}<0$. Moreover, up to a subsequence, we have asymptotic behavior
\begin{equation}\label{eqn:lim-mc-c0}
\lim_{\mu\to 0^+}I(u_{\mu,c})=\lim_{c\to 0^+}I(u_{\mu,c})= \frac{4+\alpha}{2(N+\alpha)}S^\frac{N+\alpha}{4+\alpha}_\alpha.
\end{equation}
\end{theorem}

\begin{remark}
{\rm
To achieve our purpose, the major process is as follows. We firstly apply the strategy in the pioneer work \cite{Jeanjean1997} to establish the mountain pass structure under the assumptions of Theorem \ref{Thm:normalized-bScritical-solutions-L2super+L2critical} and present the corresponding \emph{(PS)} sequence with the help of an auxiliary functional. Secondly, to overcome the difficulty caused by the HLS upper exponent, we demonstrate one alternative lemma to ensure the strong convergence of the mountain pass minimizing sequence. Once the alternative lemma is established, the crucial task is to estimate the mountain pass level within a suitable scope. Finally, taking advantage of the extremal function to the critical equation, we are allowed to reach what we want after making subtle analysis.
}
\end{remark}

The outline of this paper is organized as follows. In Section \ref{sec:preliminary}, we give some notations, use the minimax approach to construct
the mountain pass geometrical structure and reveal the existence of corresponding \emph{(PS) }sequence. In Section \ref{sec:proof-main-theorem-L2critical+L2super}, we focus our attention on the estimate of the mountain pass level and finishing proof of Theorem \ref{Thm:normalized-bScritical-solutions-L2super+L2critical}.

\section{Preliminaries}\label{sec:preliminary}
In this section, we first give some notations. Denote by ${L}^s(\mathbb{R}^N)$, $s \in [1,\infty)$, the usual Lebesgue space with its norm
$$
\|u\|_s := \left(\int_{\mathbb{R}^N}|u|^sdx\right)^\frac{1}{s}.
$$
Define the function space
$$
H^2(\mathbb{R}^N) :=\{ u\in L^2(\mathbb{R}^N)|\nabla u \in {L}^2(\mathbb{R}^N),\ \Delta u \in {L}^2(\mathbb{R}^N)\},
$$
with the usual norm
$$
\|u\|_{H^2} := \left(\int_{\mathbb{R}^N}(|\Delta u|^2 +u^2)dx\right)^\frac{1}{2}.
$$
Then, for $N \geq 5$, one has the following Gagliardo-Nirenberg inequality (\!\!\cite[Theorem in Lecture II]{Nirenberg1959})
\begin{equation}\label{eqn:GNinequality}
\|u\|_t \leq C_{N,t}\|\Delta u\|^{\gamma_t}_{2}\|u\|^{1-\gamma_t}_2, \forall u \in H^2(\mathbb{R}^N),
\end{equation}
where $t\in(2,4^*]$, $\gamma_t:=\frac{N}{2}(\frac{1}{2}-\frac{1}{t})$.

To study problem (\ref{eqn:BS-equation-L2-Super+Critical}), we also need the Hardy-Littlewood-Sobolev inequality below.
\begin{lemma}(\!\!\cite[4.3 Theorem]{Lieb-Loss2001})\label{Lem:H-L-S}
Let $r,s>1$, $0<\alpha<N$ satisfy $\frac{1}{r}+\frac{1}{s}=\frac{N+\alpha}{N}$, $f\in L^r(\mathbb R^N)$ and $g\in L^s(\mathbb R^N)$. Then there exists a sharp
constant $C(N,\alpha,r,s) > 0$, independent of $f$ and $g$, such that
\begin{equation}\label{eqn:H-L-S}
\left|\int_{\mathbb R^N}\int_{\mathbb R^N}\frac{f(x)g(y)}{|x-y|^{N-\alpha}}dxdy\right|\leq C(N,\alpha,r,s)\|f\|_r\|g\|_s.
\end{equation}
If $r=s=\frac{2N}{N+\alpha}$, then
\begin{equation}\label{eqn:Defn-CNalpha}
C(N,\alpha,r,s)=C(N,\alpha)=\pi^\frac{N-\alpha}{2}\frac{\Gamma(\frac{\alpha}{2})}{\Gamma(\frac{N+\alpha}{2})}\Big\{\frac{\Gamma(\frac{N}{2})}{\Gamma(N)}\Big\}^{-\frac{\alpha}{N}}.
\end{equation}
\end{lemma}
In view of \cite[(2.4)]{Rani-Goyal2022}, we have the best Sobolev constant $S_\alpha$ given by
\begin{equation}\label{eqn:S-alpha-define}
S_\alpha=\inf\limits_{u\in H^2(\mathbb R^N)\backslash\{0\}}\frac{\int_{\mathbb R^N}|\Delta u|^2dx}{\left(\int_{\mathbb R^N}(I_\alpha*|u|^{4^*_\alpha})|u|^{4^*_\alpha}dx\right)^\frac{1}{4^*_\alpha}}.
\end{equation}
To overcome the difficulty of lack of compactness caused by the translations, we shall consider the
functional (\ref{eqn:Defn-Iu}) on the below subset of $S(c)$
$$
S_r(c):=S(c)\cap H^2_{r}(\mathbb{R}^N)=\{u\in H^2_{r}(\mathbb{R}^N):\|u\|^2_2=c\},
$$
where
$$
H^2_{r}(\mathbb{R}^N):=\{u(x)\in H^2(\mathbb{R}^N):u(x)=u(|x|)\}.
$$
Furthermore, we denote $\mathcal P_r(c)$ by
$$
\mathcal P_r(c):=\mathcal P(c)\cap S_r(c).
$$
In order to introduce the mountain pass level for problem (\ref{eqn:BS-equation-L2-Super+Critical}), we are necessary to consider an auxiliary function space
$E:=H^2_r(\mathbb{R}^N)\times \mathbb{R}$, whose norm is defined by
$$
\|(u,s)\|_{E}:=\left(\|u\|^2_{H^2}+|s|^2_{\mathbb{R}}\right)^\frac{1}{2}
$$
and $E'$ denotes the dual space of $E$. In addition, for the arguments in the sequel conveniently, we also need the help of the continuous map $\mathcal{H}:E\to H^2_r(\mathbb{R}^N)$ with
$$
\mathcal{H}(u,s)(x):=e^{\frac{Ns}{2}}u(e^sx).
$$

In the rest of this section, we use the minimax approach to draw some conclusions.
\begin{lemma}\label{Lem:Delta-mathcalH-property}
Let $c,\mu>0$ and $\bar{p}\leq p<4^*$. In particular, for $p=\bar{p}$, assume that $c\in(0,c_0)$, where
\begin{equation}\label{eqn:defn-c0}
c_0:=\Big( \frac{\bar{p}}{8 \mu C^{\bar{p}}_{N,\bar{p}}}\Big)^\frac{N}{4}.
\end{equation}
For any fixed $u\in S_r(c)$, one has
\begin{itemize}
\item[(i)] $\|\Delta \mathcal H(u,s)\|_2\to0^+$ and $I(\mathcal H(u,s))\to0^+$ as $s\to -\infty$;
\item[(ii)] $\|\Delta \mathcal H(u,s)\|_2\to+\infty$ and $I(\mathcal H(u,s))\to-\infty$ as $s\to +\infty$.
\end{itemize}

\end{lemma}
\begin{proof}
Some direct calculations bring that
\begin{equation}\label{eqn:Delta-mathcalH-norm}
\int_{\mathbb{R}^N}|\Delta\mathcal{H}(u,s)|^2dx=e^{4s}\int_{\mathbb{R}^N}|\Delta u|^2dx,
\end{equation}
\begin{equation}\label{eqn:mathcalH-tnorm}
\int_{\mathbb{R}^N}|\mathcal{H}(u,s)|^pdx=e^{2p\gamma_p s}\int_{\mathbb{R}^N}|u|^pdx
\end{equation}
and
\begin{equation}\label{eqn:mathcalH-critical-term}
\int_{\mathbb{R}^N}(I_\alpha*|\mathcal{H}(u,s)|^{4^*_\alpha})|\mathcal{H}(u,s)|^{4^*_\alpha} dx=e^{44^*_\alpha s}\int_{\mathbb{R}^N}(I_\alpha*|u|^{4^*_\alpha})|u|^{4^*_\alpha} dx.
\end{equation}
Since $\bar{p}<p<4^*$ implies that $p\gamma_p\geq2$, it is obvious to see that
\begin{equation}\label{eqn:I-to-0-infty}
\begin{aligned}
I(\mathcal H(u,s))&=\frac{e^{4s}}{2}\int_{\mathbb{R}^N}|\Delta u|^2dx-\frac{\mu e^{2p\gamma_p s}}{p}\int_{\mathbb{R}^N}|u|^p dx-\frac{e^{44^*_\alpha s}}{24^*_\alpha}\int_{\mathbb{R}^N}(I_\alpha*|u|^{4^*_\alpha})|u|^{4^*_\alpha} dx\\
&\ \ \ \begin{cases}\to 0^+, & s\to -\infty; \\
-\infty, & s\to +\infty. \\
\end{cases}
\end{aligned}
\end{equation}
When $p=\bar{p}$, one has $\bar{p} \gamma_{\bar{p}}=2$. For this situation, together with (\ref{eqn:GNinequality}) and $c\in(0,c_0)$, we infer that
\begin{equation}\label{eqn:nabla-u-minus-bar-p-u}
\frac{1}{2} \int_{\mathbb{R}^N}|\Delta u|^2dx- \frac{\mu}{\bar{p}} \int_{\mathbb{R}^N}|u|^{\bar{p}} dx \geq \Big(\frac{1}{2}-\frac{\mu}{\bar{p}} C^{\bar{p}}_{N,\bar{p}} c^\frac{4}{N} \Big)\int_{\mathbb{R}^N}|\Delta u|^2 dx>0.
\end{equation}
It follows from $44^*_\alpha>4$ that (\ref{eqn:I-to-0-infty}) also holds for $p=\bar{p}$.
\end{proof}

\begin{lemma}\label{Prop:I-u-maximum}
Let $\bar{p}\leq p<4^*$ and assume that $c\in(0,c_0)$ for $p=\bar{p}$, where $c_0$ is given in (\ref{eqn:defn-c0}). For any $u\in S_r(c)$, there exists a unique $s_u\in\mathbb R$ such that $\mathcal H(u,s_u)\in \mathcal P_r(c)$ and $s_u$ is a strict maximum point for the functional $I(\mathcal H(u,s))$ at a positive level. Moreover, we have $s_u=0$ when $u\in \mathcal P_r(c)$.
\end{lemma}
\begin{proof}
In view of Lemma \ref{Lem:Delta-mathcalH-property}, $I(\mathcal H(u,s))$ has a global maximum point $s_u\in\mathbb R$ at a positive level.
Thus, we derive that
\begin{equation}\label{eqn:partialI-0-p}
\begin{aligned}
0&=\frac{\partial}{\partial s}I(\mathcal H(u,s_u))\\
&=2e^{4s_u}\int_{\mathbb{R}^N}|\Delta u|^2dx-2\mu\gamma_pe^{2p\gamma_p s_u}\int_{\mathbb{R}^N}|u|^p dx-2e^{44^*_\alpha s_u}\int_{\mathbb{R}^N}(I_\alpha*|u|^{4^*_\alpha})|u|^{4^*_\alpha} dx\\
&=2P(\mathcal H(u,s_u)).
\end{aligned}
\end{equation}
That is to say, $\mathcal H(u,s_u)\in \mathcal{P}_r(c)$. Hence, it remains to check the uniqueness of $s_u$. Suppose that there exist $\tilde{s}_u<s_u$ such that
$\frac{\partial}{\partial s}I(\mathcal H(u,\tilde{s}_u))=0$. Then, a simple calculation yields that
$$
\mu\gamma_p\Big(e^{(2p\gamma_p-4)s_u}-e^{(2p\gamma_p-4)\tilde{s}_u}\Big)\int_{\mathbb{R}^N}|u|^p dx+\left(e^{(44^*_\alpha-4)s_u}
-e^{(44^*_\alpha-4)\tilde{s}_u}\right)\int_{\mathbb{R}^N}(I_\alpha*|u|^{4^*_\alpha})|u|^{4^*_\alpha} dx=0,
$$
which is impossible under the assumption $\bar{p}\leq p<4^*$.
\end{proof}

\begin{lemma}\label{Lem:I-0-P-0}
Let $\bar{p}\leq p<4^*$ and assume that $c\in(0,c_0)$ for $p=\bar{p}$, where $c_0$ is given in (\ref{eqn:defn-c0}). For any $u\in S_r(c)$, if $I(u)\leq 0$, then $P(u)<0$.
\end{lemma}
\begin{proof}
From (\ref{eqn:partialI-0-p}), we see that
$I(\mathcal H(u,s))$ is strictly increasing on $(-\infty,s_u]$ and is strictly decreasing on $[s_u,+\infty)$. Together with Lemma \ref{Prop:I-u-maximum}, there holds that
$$
I(\mathcal H(u,s)) \leq I(\mathcal H(u,0))=I(u) \leq 0 \,\,\, \mbox{for}\,\, s \geq 0,
$$
which indicates that $s_u<0$. Subsequently, on account of
$$
2P\left(\mathcal H(u, s_u)\right)=\frac{\partial}{\partial s}I(\mathcal H(u,s_u))=0 \quad \mbox{and}\quad 2P(u)=2P(\mathcal H(u, 0))=\frac{\partial}{\partial s}I(\mathcal H(u,0)),
$$
we immediately draw the conclusion.
\end{proof}

\begin{lemma}\label{Lem:supAI-infBI}
Let $c,\mu>0$ and $\bar{p}\leq p<4^*$. In particular, for $p=\bar{p}$, assume that $c\in(0,c_0)$, where $c_0$ is given in (\ref{eqn:defn-c0}). For given $c>0$, we see that
\begin{equation}\label{eqn:supA-less-infB}
0<\sup\limits_{u\in \mathcal A}I(u)<\inf\limits_{u\in \mathcal B}I(u)
\end{equation}
with
\begin{equation}\label{eqn:Defn-mathcal-AB}
\mathcal A:=\{u\in S_r(c):\|\Delta u\|^2_2\leq K(c,\mu)\} \quad and \quad \mathcal B:=\{u\in S_r(c):\|\Delta u\|^2_2=2K(c,\mu)\},
\end{equation}
where
\begin{equation}\label{eqn:K-define}
\begin{aligned}
K(c,\mu)&:=\mbox{min}\Big\{\frac{S_\alpha}{2}\Big(\frac{4^*_\alpha S_\alpha}{8}\Big)^\frac{1}{4^*_\alpha-1},\ \Big(\frac{S^{4^*_\alpha}_\alpha}{4}\Big)^\frac{1}{4^*_\alpha-1},\ \Big[\frac{p}{8C^p_{N,p}2^\frac{p\gamma_p}{2}\mu c^\frac{p-p\gamma_p}{2}}\Big]^\frac{2}{p\gamma_p-2},\\
&\ \ \ \ \ \ \ \ \ \ \ \ \ \ \ \ \ \ \ \ \ \ \ \ \ \ \ \ \ \ \ \ \ \ \ \ \ \ \ \ \ \ \ \ \ \ \ \ \ \ \ \Big[\frac{p}{(p-2)NC^p_{N,p}\mu c^\frac{4p-N(p-2)}{8}}\Big]^\frac{2}{p\gamma_p-2}\Big\}.
\end{aligned}
\end{equation}
\end{lemma}
\begin{proof}
For the case $\bar{p}<p<4^*$, if $u\in \mathcal A$, $v\in \mathcal B$, we derive from (\ref{eqn:GNinequality}) and (\ref{eqn:K-define}) that
\begin{equation}\label{eqn:Iv-minus-Iu-AB}
\begin{aligned}
I(v)-I(u)&=\frac{1}{2}\int_{\mathbb{R}^N}|\Delta v|^2 dx-\frac{\mu}{p}\int_{\mathbb{R}^N}|v|^p dx-\frac{1}{2 4^*_\alpha}\int_{\mathbb{R}^N}(I_\alpha*|v|^{4^*_\alpha})|v|^{4^*_\alpha} dx-\frac{1}{2}\int_{\mathbb{R}^N}|\Delta u|^2 dx\\
&\ \ \ +\frac{\mu}{p}\int_{\mathbb{R}^N}|u|^p dx+\frac{1}{24^*_\alpha}\int_{\mathbb{R}^N}(I_\alpha*|u|^{4^*_\alpha})|u|^{4^*_\alpha} dx\\
&\geq \frac{1}{2}\int_{\mathbb{R}^N}|\Delta v|^2 dx-\frac{1}{2}\int_{\mathbb{R}^N}|\Delta u|^2 dx-\frac{\mu}{p}\int_{\mathbb{R}^N}|v|^p dx-\frac{1}{2 4^*_\alpha}\int_{\mathbb{R}^N}(I_\alpha*|v|^{4^*_\alpha})|v|^{4^*_\alpha} dx\\
&\geq\frac{1}{2}K(c,\mu)\!-\frac{\mu}{p}C^p_{N,p}2^\frac{p\gamma_p}{2}c^\frac{p-p\gamma_p}{2}\Big(K(c,\mu)\Big)^\frac{p\gamma_p}{2}\!
-\!\frac{2^{4^*_\alpha-1}}{4^*_\alpha S^{4^*_\alpha}_\alpha}\!\Big(K(c,\mu)\Big)^{4^*_\alpha}\\
&=\frac{1}{4}K(c,\mu)\!-\frac{\mu}{p}C^p_{N,p}2^\frac{p\gamma_p}{2}c^\frac{p-p\gamma_p}{2}\Big(K(c,\mu)\Big)^\frac{p\gamma_p}{2}+\frac{1}{4}K(c,\mu)\!
-\!\frac{2^{4^*_\alpha-1}}{4^*_\alpha S^{4^*_\alpha}_\alpha}\!\Big(K(c,\mu)\Big)^{4^*_\alpha}\\
&> 0.
\end{aligned}
\end{equation}

When $0<c<c_0$, note that
$$
\frac{1}{2}K(c,\mu)-\frac{\mu}{\bar{p}}C^{\bar{p}}_{N,\bar{p}}2^\frac{\bar{p}\gamma_{\bar{p}}}{2}c^\frac{\bar{p}-\bar{p}\gamma_{\bar{p}}}{2}\Big(K(c,\mu)\Big)^\frac{\bar{p}\gamma_{\bar{p}}}{2}
=\Bigl(\frac{1}{2}-\frac{\mu}{\bar{p}}C^{\bar{p}}_{N,\bar{p}}2c^\frac{4}{N}\Bigr)K(c,\mu)
\geq \frac{1}{4}K(c,\mu).
$$
Then, the above inequality (\ref{eqn:Iv-minus-Iu-AB}) also holds for $p=\bar{p}$, $u\in \mathcal{A}$ and $v\in \mathcal{B}$. That is, (\ref{eqn:supA-less-infB}) is verified.
\end{proof}

\begin{lemma}\label{Cor:Iu-0-I*}
Let $c,\mu>0$ and $\bar{p}\leq p<4^*$. In particular, for $p=\bar{p}$, assume that $c\in(0,c_0)$, where $c_0$ is given in (\ref{eqn:defn-c0}).
For $K(c,\mu)>0$ given in (\ref{eqn:K-define}), there holds that $I(u)>0$ for any $u\in S_r(c)$ with $\|\Delta u\|^2_2\leq K(c,\mu)$. Moreover,
$$
I_*:=\inf \left\{I(u):u\in S_r(c)\ and \ \|\Delta u\|^2_2=\frac{K(c,\mu)}{2}\right\}>0.
$$
\end{lemma}
\begin{proof}
From (\ref{eqn:GNinequality}), we have
$$
\begin{aligned}
I(u)&=\frac{1}{2}\int_{\mathbb R^N}|\Delta u|^2dx-\frac{\mu}{p}\int_{\mathbb R^N}|u|^pdx-\frac{1}{24^*_\alpha}\int_{\mathbb{R}^N}(I_\alpha*|u|^{4^*_\alpha})|u|^{4^*_\alpha} dx\\
&\geq\frac{1}{2}\|\Delta u\|^2_2-\frac{\mu}{p}C^p_{N,p}c^\frac{p-p\gamma_p}{2}\Big(\|\Delta u\|^2_2\Big)^\frac{p\gamma_p}{2}-\frac{1}{24^*_\alpha S^{4^*_\alpha}_\alpha}\Big(\|\Delta u\|^2_2\Big)^{4^*_\alpha}\\
&=\|\Delta u\|^2_2\Big[\frac{1}{2}-\frac{\mu}{p}C^p_{N,p}c^\frac{p-p\gamma_p}{2}\Big(\|\Delta u\|^2_2\Big)^\frac{p\gamma_p-2}{2}-\frac{1}{24^*_\alpha S^{4^*_\alpha}_\alpha}\Big(\|\Delta u\|^2_2\Big)^{4^*_\alpha-1}\Big]\\
&=\|\Delta u\|^2_2f(c,\|\Delta u\|^2_2),
\end{aligned}
$$
where
$$
f(c,r):=\frac{1}{2}-\frac{\mu}{p}C^p_{N,p}c^\frac{p-p\gamma_p}{2}r^\frac{p\gamma_p-2}{2}-\frac{1}{24^*_\alpha S^{4^*_\alpha}_\alpha}r^{4^*_\alpha-1},\ \ \ r\geq 0.
$$
For any $u\in S_r(c)$ with $\|\Delta u\|^2_2\leq K(c,\mu)$, observe that
$$
\begin{aligned}
f(c,\|\Delta u\|^2_2)&=\frac{1}{2}-\frac{\mu}{p}C^p_{N,p}c^\frac{p-p\gamma_p}{2}\Big(\|\Delta u\|^2_2\Big)^\frac{p\gamma_p-2}{2}-\frac{1}{24^*_\alpha S^{4^*_\alpha}_\alpha}\Big(\|\Delta u\|^2_2\Big)^{4^*_\alpha-1}\\
&\geq\frac{1}{2}-\frac{\mu}{p}C^p_{N,p}c^\frac{p-p\gamma_p}{2}\Big(K(c,\mu)\Big)^\frac{p\gamma_p-2}{2}\!
-\!\frac{1}{24^*_\alpha S^{4^*_\alpha}_\alpha}\Big(K(c,\mu)\Big)^{4^*_\alpha-1}\\
&\geq\frac{1}{2}\!-\frac{\mu}{p}C^p_{N,p}2^\frac{p\gamma_p}{2}c^\frac{p-p\gamma_p}{2}\Big(K(c,\mu)\Big)^\frac{p\gamma_p-2}{2}\!
-\!\frac{2^{4^*_\alpha-1}}{4^*_\alpha S^{4^*_\alpha}_\alpha}\Big(K(c,\mu)\Big)^{4^*_\alpha-1}\\
&>0.
\end{aligned}
$$
As a consequence, $I(u)>0$, $\forall u\in S_r(c)$ with $\|\Delta u\|^2_2\leq K(c,\mu)$.

If $\|\Delta u\|^2_2=\frac{K(c,\mu)}{2}$, according to the definition of $K(c,\mu)$, we readily infer that
$$
\begin{aligned}
I(u)&\geq\frac{1}{2}\|\Delta u\|^2_2-\frac{\mu}{p}C^p_{N,p}c^\frac{p-p\gamma_p}{2}\Big(\|\Delta u\|^2_2\Big)^\frac{p\gamma_p}{2}-\frac{1}{24^*_\alpha S^{4^*_\alpha}_\alpha}\Big(\|\Delta u\|^2_2\Big)^{4^*_\alpha}\\
&=\frac{K(c,\mu)}{2}\Big[\frac{1}{2}-\frac{\mu}{p}C^p_{N,p}c^\frac{p-p\gamma_p}{2}\Big(\frac{K(c,\mu)}{2}\Big)^\frac{p\gamma_p-2}{2}
-\frac{1}{24^*_\alpha S^{4^*_\alpha}_\alpha}\Big(\frac{K(c,\mu)}{2}\Big)^{4^*_\alpha-1}\Big]\\
&> \frac{K(c,\mu)}{2}\Big[\frac{1}{4}-\frac{\mu}{p}C^p_{N,p}c^\frac{p-p\gamma_p}{2}\Big(K(c,\mu)\Big)^\frac{p\gamma_p-2}{2}+\frac{1}{4}
-\frac{1}{4^*_\alpha S^{4^*_\alpha}_\alpha}\Big(K(c,\mu)\Big)^{4^*_\alpha-1}\Big]\\
&>0,
\end{aligned}
$$
which states that $I_*>0$.
\end{proof}

\begin{remark}\label{rem:MP-level}
{\rm
For any given $\hat{u}\in S_r(c)$, take into account Lemmas \ref{Lem:Delta-mathcalH-property}, \ref{Lem:supAI-infBI} and \ref{Cor:Iu-0-I*},
obviously there are two numbers
\begin{equation}\label{eqn:Defn-s1-s2}
s_1:=s_1(c,\mu,\hat{u})<0 \quad {\rm and} \quad  s_2:=s_2(c,\mu,\hat{u})>0
\end{equation}
to guarantee that the functions
\begin{equation}\label{eqn:Defn-hat-u1-u2}
\hat{u}_{1,\mu}:=\mathcal H(\hat{u},s_1) \quad {\rm and} \quad \hat{u}_{2,\mu}:=\mathcal H(\hat{u},s_2)
\end{equation}
satisfy
$$
\|\Delta \hat{u}_{1,\mu}\|^2_2<\frac{K(c,\mu)}{2},\ \ \|\Delta \hat{u}_{2,\mu}\|^2_2>2K(c,\mu),\ \ I(\hat{u}_{1,\mu})>0\ \ \mbox{and}\ \ I(\hat{u}_{2,\mu})<0.
$$
Therefore, following the idea from Jeanjean \cite{Jeanjean1997}, we can determine the following mountain pass level denoted by
\begin{equation}\label{eqn:Defn-gamma-mu}
\gamma_\mu(c):=\inf\limits_{h\in\Gamma(c)}\max\limits_{t\in[0,1]}I(h(t))
\end{equation}
with
\begin{equation}\label{eqn:Defn-Gamma-c}
\Gamma(c):=\Big\{h\in C([0,1],S_r(c)):\|\Delta h(0)\|^2_2<\frac{K(c,\mu)}{2}\ \mbox{and}\ I(h(1))<0\Big\}.
\end{equation}
To guarantee the existence of \emph{(PS)} sequence for $I$ at the level $\gamma_\mu(c)$ defined in (\ref{eqn:Defn-gamma-mu}), we need the help of an auxiliary functional $\tilde{I}:E=H^2(\mathbb{R}^N\times \mathbb{R})\to \mathbb R$ given below
\begin{equation}\label{eqn:Defn-tilde-I}
\tilde{I}(u,s):=\frac{e^{4s}}{2}\int_{\mathbb{R}^N}|\Delta u|^2dx-\frac{\mu e^{2p\gamma_p s}}{p}\int_{\mathbb{R}^N}|u|^p dx-\frac{e^{44^*_\alpha s}}{24^*_\alpha}\int_{\mathbb{R}^N}(I_\alpha*|u|^{4^*_\alpha})|u|^{4^*_\alpha} dx
=I(\mathcal H(u,s)).
\end{equation}
As the functional $\tilde{I}$ is concerned, it is necessary to consider the minimax level
\begin{equation}\label{eqn:Defn-widetilde-gamma}
\widetilde{\gamma}_\mu(c):=\inf\limits_{\tilde{h}\in\widetilde{\Gamma}(c)}\max\limits_{t\in[0,1]}\tilde{I}(\tilde{h}(t)),
\end{equation}
where the path family $\tilde{\Gamma}(c)$ is supposed to satisfy the following requirement
\begin{equation}\label{eqn:Defn-widetilde-Gamma}
\begin{aligned}
\widetilde{\Gamma}(c):=\Big\{\tilde{h}=&(\tilde{h}_1,\tilde{h}_2)\in C([0,1],S_r(c)\times \mathbb{R}):\tilde{h}(0)=(\tilde{h}_{1}(0),0),\ \ \tilde{h}(1)=(\tilde{h}_{1}(1),0),\\
&\|\Delta \tilde{h}_{1}(0)\|^2_2<\frac{K(c,\mu)}{2} \ \ \mbox{and}\ \ \tilde{I}(\tilde{h}_{1}(1),0)<0\Big\}.
\end{aligned}
\end{equation}
}
\end{remark}

\begin{lemma}\label{Lem:gamma-tildegamma-equal}
Let $c,\mu>0$ and $\bar{p}\leq p<4^*$. In particular, for $p=\bar{p}$, assume that $c\in(0,c_0)$, where $c_0$ is given in (\ref{eqn:defn-c0}). Then, we have
\begin{equation}\label{eqn:gamma-mu-widegamma-mu}
\gamma_\mu(c)=\widetilde{\gamma}_\mu(c)\geq I_*>0.
\end{equation}
\end{lemma}

\begin{proof}
In view of Lemma \ref{Cor:Iu-0-I*}, we see that $\|\Delta h(1)\|^2_2>K(c,\mu)$ for each $h\in\Gamma(c)$. Note that $\|\Delta h(0)\|^2_2<\frac{K(c,\mu)}{2}$,
there must be some $t_0\in(0,1)$ such that $\|\Delta h(t_0)\|^2_2=\frac{K(c,\mu)}{2}$. Hence, it indicates that
$$
\max\limits_{t\in[0,1]}I(h(t))\geq I(h(t_0))\geq I_*>0.
$$
That is, $\gamma_\mu(c)\geq I_*>0$.

For any $\tilde{h}(t)\in\widetilde{\Gamma}(c)$, it can be rewritten as $\tilde{h}(t)=(\tilde{h}_1(t),\tilde{h}_2(t))\in S_r(c)\times \mathbb R$. Set
$h(t)=\mathcal H (\tilde{h}_1(t),\tilde{h}_2(t))$, then $h(t)\in \Gamma(c)$ and
$$
\max\limits_{t\in[0,1]}\tilde{I}(\tilde{h}(t))=\max\limits_{t\in[0,1]}I(h(t))\geq \gamma_\mu(c),
$$
which implies that $\widetilde{\gamma}_\mu(c)\geq\gamma_\mu(c)$. On the other hand, for any $h\in\Gamma(c)$, denote $\tilde{h}(t)=(h(t),0)$, then $\tilde{h}(t)\in\widetilde{\Gamma}(c)$ and
$$
\widetilde{\gamma}_\mu(c)\leq\max\limits_{t\in[0,1]}\tilde{I}(\tilde{h}(t))
=\max\limits_{t\in[0,1]}I(h(t)),
$$
it means that $\gamma_\mu(c)\geq\widetilde{\gamma}_\mu(c)$. Therefore, we conclude that $\widetilde{\gamma}_\mu(c)=\gamma_\mu(c)\geq I_*>0$.
\end{proof}

\begin{lemma}\label{Lem:gammamu-infIu}
Let $\bar{p}\leq p<4^*$ and assume that $c\in(0,c_0)$ for $p=\bar{p}$, where $c_0$ is given in (\ref{eqn:defn-c0}). Then we conclude that $\gamma_\mu(c)=m(c):=\inf\limits_{u\in \mathcal P_r(c)}I(u)$.
\end{lemma}
\begin{proof}
For any $u\in\mathcal P_r(c)$, Lemmas \ref{Lem:Delta-mathcalH-property} and \ref{Lem:supAI-infBI} imply the existence of two constants $s_1=s_1(c,\mu,u)<0$ and
$s_2=s_2(c,\mu,u)>0$ to guarantee that $h(t):=\mathcal H(u,(1-t)s_1+ts_2)$, $t\in[0,1]$, belongs to $\Gamma(c)$,
Therefore, it follows from Lemma \ref{Prop:I-u-maximum} that
$$
\gamma_\mu(c)\leq\max\limits_{t\in[0,1]}I(h(t))=\max\limits_{t\in[0,1]}I(\mathcal H(u,(1-t)s_1+ts_2))\leq I(u),
$$
namely, $\gamma_\mu(c)\leq m(c)$.

On the other hand, for any $\tilde{h}(t):=\Big(\tilde{h}_1(t), \tilde{h}_2(t)\Big) \in \widetilde{\Gamma}(c)$, we consider the function
$$
\widetilde{P}(t):=P\Big(\mathcal H(\tilde{h}_1(t), \tilde{h}_2(t))\Big).
$$
From (\ref{eqn:Pohozaev-identify}), (\ref{eqn:GNinequality}) and (\ref{eqn:S-alpha-define}), we have
\begin{equation}\label{eqn:P-leq-g}
\begin{aligned}
P(u)&=\|\Delta u\|^2_2-\mu\gamma_p\|u\|^p_p-\int_{\mathbb{R}^N}(I_\alpha*|u|^{4^*_\alpha})|u|^{4^*_\alpha} dx\\
&\geq\|\Delta u\|^2_2-\mu\gamma_p C_{N,p}^p c^\frac{p-p\gamma_p}{2}(\|\Delta u\|^2_2)^{p\gamma_p}-\frac{1}{S^{4^*_\alpha}_\alpha}(\|\Delta u\|^2_2)^{4^*_\alpha}\\
&=\|\Delta u\|^2_2\Big(1-\mu\gamma_p C_{N,p}^p c^\frac{p-p\gamma_p}{2}(\|\Delta u\|^2_2)^\frac{p\gamma_p-2}{2}-\frac{1}{S^{4^*_\alpha}_\alpha}(\|\Delta u\|^2_2)^{4^*_\alpha-1}\Big)\\
&=\|\Delta u\|^2_2 g(c,\|\Delta u\|^2_2),
\end{aligned}
\end{equation}
where
$$
g(c,r):=1-\mu\gamma_p C_{N,p}^p c^\frac{p-p\gamma_p}{2}r^\frac{p\gamma_p-2}{2}-\frac{1}{S^{4^*_\alpha}_\alpha}r^{4^*_\alpha-1},\ \ \ r\geq 0.
$$
For any $u\in \mathcal A$ defined in (\ref{eqn:Defn-mathcal-AB}), from the definition of $K(c,\mu)>0$ in (\ref{eqn:K-define}), we see that
$$
\begin{aligned}
g(c,\|\Delta u\|^2_2)&=1-\mu\gamma_p C_{N,p}^p c^\frac{p-p\gamma_p}{2}(\|\Delta u\|^2_2)^\frac{p\gamma_p-2}{2}-\frac{1}{S^{4^*_\alpha}_\alpha}(\|\Delta u\|^2_2)^{4^*_\alpha-1}\\
&\geq1-\mu\gamma_p C_{N,p}^p c^\frac{p-p\gamma_p}{2}(K(c,\mu))^\frac{p\gamma_p-2}{2}-\frac{1}{S^{4^*_\alpha}_\alpha}(K(c,\mu))^{4^*_\alpha-1}\\
&>0,
\end{aligned}
$$
which brings that
\begin{equation}\label{eqn:P-geq-0}
P(u)>0, \ \ \ \forall u\in \mathcal A.
\end{equation}
Thereby, utilizing (\ref{eqn:P-geq-0}) and Lemma \ref{Lem:I-0-P-0}, we infer that
$$
\widetilde{P}(0)=P\left(\mathcal H(\tilde{h}_1(0), \tilde{h}_2(0))\right)=P\left(\tilde{h}_1(0)\right)>0
$$
and
$$
\widetilde{P}(1)=P\left(\mathcal H(\tilde{h}_1(1), \tilde{h}_2(1))\right)=P\left(\tilde{h}_1(1)\right)<0.
$$

Take advantage of the continuity of $\widetilde{P}_r(t)$, there exists some $\tilde{t} \in(0,1)$ such that
$\widetilde{P}(\tilde{t})=0$, which means that $\mathcal H(\tilde{h}_1(\tilde{t}), \tilde{h}_2(\tilde{t})) \in \mathcal P_r(c)$. Consequently, one has
$$
\max_{t \in[0,1]} \tilde{I}(\tilde{h}(t))=\max_{t \in[0,1]} I\left(\mathcal H(\tilde{h}_1(t), \tilde{h}_2(t))\right) \geq \inf_{u \in \mathcal P_r(c)}I(u).
$$
Due to the arbitrariness of $\tilde{h}(t)$, it signifies that $\widetilde{\gamma}_\mu(c)=\gamma_\mu(c) \geq m(c)$. Hence, $\gamma_\mu(c) = m(c)$.
\end{proof}

Based on Lemma \ref{Lem:supAI-infBI}, Lemma \ref{Lem:gamma-tildegamma-equal} and the definition of $\tilde{I}$, it indicates that
\begin{equation}\label{eqn:gamma-mu-c-big-gamma-0-mu-c}
\begin{aligned}
\widetilde{\gamma}_\mu(c)&=\gamma_\mu(c)>\sup_{\tilde{h}\in \widetilde{\Gamma}(c)}\{I(h_1(0)),I(h_1(1))\}=\sup_{\tilde{h}\in \widetilde{\Gamma}(c)}\{\tilde{I}(\tilde{h}(0)),\tilde{I}(\tilde{h}(1))\}=:(\widetilde{\gamma}_\mu)_0(c).
\end{aligned}
\end{equation}

\begin{lemma}\label{Lem:three-tildeJ-prove}
Let $\bar{p}\leq p<4^*$ and assume that $c\in(0,c_0)$ for $p=\bar{p}$, where $c_0$ is given in (\ref{eqn:defn-c0}). For $0<\varepsilon<\widetilde{\gamma}_\mu(c)-(\widetilde{\gamma}_\mu)_0(c)$, choose $g\in \widetilde{\Gamma}(c)$ such that $\max\limits_{t\in[0,1]}\tilde{I}(g(t)) \leq \widetilde{\gamma}_\mu(c)+\varepsilon$. Then there exists $(v,s)\in S_r(c)\times \mathbb R$ to guarantee that
\begin{itemize}
\item[(i)]  $\tilde{I}(v,s)\in\left[\widetilde{\gamma}_\mu(c)-\varepsilon,\widetilde{\gamma}_\mu(c)+\varepsilon\right]$;
\item[(ii)] $\min\limits_{t\in[0,1]}\|(v,s)-g(t)\|_E\leq \sqrt{\varepsilon}$;
\item[(iii)]$\|(\tilde{I}|_{S_r(c)\times\mathbb R})'(v,s)\|_{E'}\leq 2\sqrt{\varepsilon}$,
\end{itemize}
where $\widetilde{\gamma}_\mu(c)$, $\widetilde{\Gamma}(c)$ and $(\widetilde{\gamma}_\mu)_0(c)$are given in (\ref{eqn:Defn-widetilde-gamma}), (\ref{eqn:Defn-widetilde-Gamma}) and (\ref{eqn:gamma-mu-c-big-gamma-0-mu-c}), respectively.
\end{lemma}
\begin{proof} Its proof is almost the repetition of the process of \cite[Lemma 2.5]{Liu-Zhang2023}, so we omit the details.

\end{proof}

The following proposition can be obtained directly from the above lemma.

\begin{proposition}\label{Pro:three-tildeJ}
Let $\bar{p}\leq p<4^*$ and assume that $c\in(0,c_0)$ for $p=\bar{p}$, where $c_0$ is given in (\ref{eqn:defn-c0}). Choose $\{g_n\}\subset\widetilde{\Gamma}(c)$ such that $\max\limits_{t\in[0,1]}\tilde{I}(g_n(t))\leq \widetilde{\gamma}_\mu(c)+\frac{1}{n}$, then there exists a sequence
$\{(v_n,s_n)\}\subset S_r(c)\times \mathbb{R}$ such that
\begin{itemize}
\item[(i)]  $\tilde{I}(v_n,s_n)\in\left[\widetilde{\gamma}_\mu(c)-\frac{1}{n},\widetilde{\gamma}_\mu(c)+\frac{1}{n}\right]$;
\item[(ii)] $\min\limits_{t\in[0,1]}\|(v_n,s_n)-g_n(t)\|_E\leq \frac{1}{\sqrt{n}}$;
\item[(iii)]$\|(\tilde{I}|_{S_r(c)\times\mathbb R})'(v_n,s_n)\|_{E'}\leq \frac{2}{\sqrt{n}}$, that is,
$$
|\langle \tilde{I}'(v_n,s_n),\omega \rangle_{E'\times E}|
\leq \frac{2}{\sqrt{n}}\|\omega\|_E
$$ for all $\omega \in \widetilde{T}_{(v_n,s_n)}:=\{(\omega_1,\omega_2)\in E,\langle v_n,\omega_1 \rangle_{L^2}=0\}$.
\end{itemize}
\end{proposition}

\begin{lemma}\label{Lem:PS-true}
Let $\bar{p}\leq p<4^*$ and assume that $c\in(0,c_0)$ for $p=\bar{p}$, where $c_0$ is given in (\ref{eqn:defn-c0}). For the sequence $\{(v_n,s_n)\}$ obtained in Proposition \ref{Pro:three-tildeJ}, set $u_n:=\mathcal H(v_n,s_n)=e^{\frac{Ns_n}{2}}v_n(e^{s_n} x)$,
then there hold that
\begin{itemize}
\item[(i)]  $I(u_n)\to\gamma_\mu(c)$ as $n\to\infty$;
\item[(ii)] $\partial_s \tilde{I}(v_n,s_n)=2\|\Delta u_n\|^2_2-2\mu \gamma_p\|u_n\|^p_p-2\int_{\mathbb R^N}(I_\alpha*|u_n|^{4^*_\alpha})|u_n|^{4^*_\alpha}dx\to 0$
as $n\to\infty$;
\item[(iii)] $\|u_n\|_{H^2}$ and $\frac{\mu}{p}\|u_n\|^p_p+\frac{1}{24^*_\alpha}\int_{\mathbb R^N}(I_\alpha*|u_n|^{4^*_\alpha})|u_n|^{4^*_\alpha}dx$ are bounded;
\item[(iv)] $|\langle I'(u_n),\omega \rangle_{(H^2(\mathbb{R}^N))'\times H^2(\mathbb{R}^N)}|\leq \frac{2e^4}{\sqrt{n}}\|\omega\|_{H^2}$ for all
$\omega\in T_{u_n}\!\!:=\{\omega\in H^2(\mathbb R^N),\langle u_n,\omega \rangle_{L^2}=0\}$.
\end{itemize}
\end{lemma}
\begin{proof}
Observe that $I(u_n)=\tilde{I}(v_n,s_n)$ and $\widetilde{\gamma}_\mu(c)=\gamma_\mu(c)$, \emph{(i)} is a direct conclusion from Proposition \ref{Pro:three-tildeJ}-\emph{(i)}.

\emph{(ii)} Letting $\partial_s \tilde{I}(v_n,s_n)=\langle \tilde{I}'(v_n,s_n),(0,1)\rangle_{E'\times E}$, in view of Proposition \ref{Pro:three-tildeJ}-\emph{(iii)}, we see that
$$
\partial_s \tilde{I}(v_n,s_n)\to 0 \ \ \mbox{as}\ \  n\to \infty.
$$
More precisely, a simple calculation indicates that
$$
\partial_s \tilde{I}(v_n,s_n)=2\|\Delta u_n\|^2_2-2\mu \gamma_p\|u_n\|^p_p-2\int_{\mathbb R^N}(I_\alpha*|u_n|^{4^*_\alpha})|u_n|^{4^*_\alpha}dx\to 0\ \mbox{as}\ n\to \infty.
$$

\emph{(iii)} Note the boundedness of $\tilde{I}(v_n,s_n)$, it follows that
\begin{equation}\label{eqn:inrquality-2}
-C\leq\tilde{I}(v_n,s_n)=I(u_n)=\frac{1}{2}\|\Delta u_n\|^2_2-\left[\frac{\mu}{p}\|u_n\|^p_p+\frac{1}{24^*_\alpha}\int_{\mathbb R^N}(I_\alpha*|u_n|^{4^*_\alpha})|u_n|^{4^*_\alpha}dx\right]\leq C\\
\end{equation}
for some constant $C>0$. When $\bar{p}<p<4^*$, the boundedness of $\tilde{I}(v_n,s_n)$ and $\partial_s \tilde{I}(v_n,s_n)$ also brings that
(modifying $C>0$ if necessary)
 \begin{equation}\label{eqn:inrquality-1}
\begin{aligned}
-C&\leq N\tilde{I}(v_n,s_n)+\partial_s \tilde{I}(v_n,s_n)\\
&=\frac{N}{2}\|\Delta u_n\|^2_2-\frac{\mu N}{p}\|u_n\|^p_p-\frac{N}{24^*_\alpha}\int_{\mathbb R^N}(I_\alpha*|u_n|^{4^*_\alpha})|u_n|^{4^*_\alpha}dx\\
&\ \ \ +2\|\Delta u_n\|^2_2-2\mu \gamma_p\|u_n\|^p_p-2\int_{\mathbb R^N}(I_\alpha*|u_n|^{4^*_\alpha})|u_n|^{4^*_\alpha}dx\\
&=\frac{N}{2}\|\Delta u_n\|^2_2-\frac{\mu N}{p}\|u_n\|^p_p-\frac{N}{24^*_\alpha}\int_{\mathbb R^N}(I_\alpha*|u_n|^{4^*_\alpha})|u_n|^{4^*_\alpha}dx\\
&\ \ \ +2\|\Delta u_n\|^2_2-\mu N(\frac{1}{2}-\frac{1}{p})\|u_n\|^p_p-2\int_{\mathbb R^N}(I_\alpha*|u_n|^{4^*_\alpha})|u_n|^{4^*_\alpha}dx\\
&\leq\frac{N+4}{2}\|\Delta u_n\|^2_2-\frac{Np}{2}\left[\frac{\mu}{p}\|u_n\|^p_p+\frac{1}{24^*_\alpha}\int_{\mathbb R^N}(I_\alpha*|u_n|^{4^*_\alpha})|u_n|^{4^*_\alpha}dx\right],
\end{aligned}
\end{equation}
since $\frac{N}{24^*_\alpha}+2-\frac{Np}{44^*_\alpha}>0$. Combining (\ref{eqn:inrquality-2}) and (\ref{eqn:inrquality-1}), we deduce that
$$
(N+4-\frac{Np}{2})\left[\frac{\mu}{p}\|u_n\|^p_p+\frac{1}{24^*_\alpha}\int_{\mathbb R^N}(I_\alpha*|u_n|^{4^*_\alpha})|u_n|^{4^*_\alpha}dx\right]\geq -(N+5)C.
$$
Since $2+\frac{8}{N}<p<4^*$, the above inequality ensures the boundedness of $\frac{\mu}{p}\|u_n\|^p_p+\frac{1}{24^*_\alpha}\int_{\mathbb R^N}(I_\alpha*|u_n|^{4^*_\alpha})|u_n|^{4^*_\alpha}dx$ and consequently $\|\Delta u_n\|^2_2$ is also.

For the case $p=\bar{p}$, by $(ii)$, we know that
$$
\|\Delta u_n\|^2_2-\mu \gamma_{\bar {p}}\|u_n\|^{\bar {p}}_{\bar{p}}-\int_{\mathbb R^N}(I_\alpha*|u_n|^{4^*_\alpha})|u_n|^{4^*_\alpha}dx=o_n(1).
$$
Together this with (\ref{eqn:inrquality-2}) and (\ref{eqn:GNinequality}), we deduce that
$$
\begin{aligned}
C&\geq\frac{1}{2}\|\Delta u_n\|^2_2-\frac{\mu}{\bar{p}}\|u_n\|^{\bar{p}}_{\bar{p}}-\frac{1}{24^*_\alpha}\int_{\mathbb R^N}(I_\alpha*|u_n|^{4^*_\alpha})|u_n|^{4^*_\alpha}dx\\
&=\Big(\frac{1}{2}-\frac{1}{24^*_\alpha}\Big)\|\Delta u_n\|^2_2-\Big(\frac{1}{2}-\frac{1}{24^*_\alpha}\Big)\frac{2\mu}{\bar{p}}\|u_n\|^{\bar{p}}_{\bar{p}}+o_n(1)\\
&\geq \Big(\frac{1}{2}-\frac{1}{24^*_\alpha}\Big)\|\Delta u_n\|^2_2-\Big(\frac{1}{2}-\frac{1}{24^*_\alpha}\Big)\frac{2\mu}{\bar{p}}c^\frac{4}{N}C^{\bar{p}}_{N,\bar{p}}\|\Delta u_n\|^2_2+o_n(1)\\
&= \Big(\frac{1}{2}-\frac{1}{24^*_\alpha}\Big)\Big(1-\frac{2\mu}{\bar{p}}c^\frac{4}{N}C^{\bar{p}}_{N,\bar{p}}\Big)\|\Delta u_n\|^2_2+o_n(1),
\end{aligned}
$$
which implies that $\{u_n\}$ is bounded in $H^2(\mathbb R^N)$. Thus, the boundedness of $\frac{\mu}{\bar{p}}\|u_n\|^{\bar{p}}_{\bar{p}}+\frac{1}{24^*_\alpha}\int_{\mathbb R^N}(I_\alpha*|u_n|^{4^*_\alpha})|u_n|^{4^*_\alpha}dx$ are obtained from (\ref{eqn:inrquality-2}).

\emph{(iv)} For $h_n\in T_{u_n}$ and $\tilde{h}_n(x):=e^{-\frac{Ns_n}{2}}h_n(e^{-s_n}x)$, we have
$$
\begin{aligned}
\langle &I'(u_n),h_n \rangle_{(H^2)'\times H^2}\\
&=e^\frac{(4+N)s_n}{2}\int_{\mathbb R^N}\!\!\!\!\!\!\Delta v_n(e^{s_n}x) \Delta h_n(x) dx-\mu e^\frac{N(p-1)s_n}{2}\int_{\mathbb R^N}|v_n(e^{s_n}x)|^{p-2}v_n(e^{s_n}x) h_n(x)dx\\
&\ \ \ - e^\frac{N(24^*_\alpha-1)s_n}{2}\int_{\mathbb R^N}(I_\alpha*|v_n(e^{s_n}y)|^{4^*_\alpha})|v_n(e^{s_n}x)|^{4^*_\alpha-2}v_n(e^{s_n}x) h_n(x)dx\\
&=\!e^{4s_n}\!\!\int_{\mathbb R^N}\!\!\!\!\!\!\!\Delta v_n(x) \Delta \Big(h_n\!(e^{-s_n}x)e^{-\frac{Ns_n}{2}}\Big)dx\!-\!\!\mu e^\frac{N(p-2)s_n}{2}\!\!\!\!
\int_{\mathbb R^N}\!\!\!\!\!\!|v_n(x)|^{p-2}\!v_n(x) h_n\!(e^{-s_n}x)e^{-\frac{Ns_n}{2}}\!dx\\
&\ \ \ - e^{44^*_\alpha s_n}\int_{\mathbb R^N}(I_\alpha*|v_n(y)|^{4^*_\alpha})|v_n(x)|^{4^*_\alpha-2}v_n(x) h_n(e^{-s_n}x)e^{-\frac{Ns_n}{2}}dx\\
&=\!e^{4s_n}\!\!\int_{\mathbb R^N}\!\!\!\!\!\!\!\Delta v_n(x) \Delta \Big(h_n\!(e^{-s_n}x)e^{-\frac{Ns_n}{2}}\Big)dx\!-\!\!\mu e^{2p\gamma_p s_n}\!\!\!\!
\int_{\mathbb R^N}\!\!\!\!\!\!|v_n(x)|^{p-2}\!v_n(x) h_n\!(e^{-s_n}x)e^{-\frac{Ns_n}{2}}\!dx\\
&\ \ \ - e^{44^*_\alpha s_n}\int_{\mathbb R^N}(I_\alpha*|v_n(y)|^{4^*_\alpha})|v_n(x)|^{4^*_\alpha-2}v_n(x) h_n(e^{-s_n}x)e^{-\frac{Ns_n}{2}}dx\\
&=\langle \tilde{I}'(v_n,s_n),(\tilde{h}_n,0) \rangle_{E'\times E}.\\
\end{aligned}
$$
In addition, since $h_n\in T_{u_n}$, that is, $\langle u_n,h_n \rangle_{L^2}=0$, it gives that
$$
0=\int_{\mathbb R^N}e^\frac{Ns_n}{2}v_n(e^{s_n}x)h_n(x)dx=\int_{\mathbb R^N}v_n(x)e^{-\frac{Ns_n}{2}}h_n(e^{-s_n}x)dx=\langle v_n, \tilde{h}_n\rangle_{L^2}.
$$
Hence, $(\tilde{h}_n,0)\in \widetilde{T}_{(v_n,s_n)}$, where $\widetilde{T}_{(v_n,s_n)}$ is given in Proposition \ref{Pro:three-tildeJ}. Using this fact and Proposition
\ref{Pro:three-tildeJ}-\emph{(iii)}, we see that
$$
|\langle I'(u_n),h_n \rangle_{(H^2)'\times H^2}|=|\langle \tilde{I}'(v_n,s_n),(\tilde{h}_n,0) \rangle_{E'\times E}|\leq \frac{2}{\sqrt{n}}\|(\tilde{h}_n,0)\|_{E}.
$$

Therefore, it only remains to verify that $\|(\tilde{h}_n,0)\|_{E}\leq e^4\|h_n\|_{H^2}$ for all $n$. Indeed, since $\widetilde{\gamma}_\mu(c)=\gamma_\mu(c)$, there exists
$\{g_n\}\subset \widetilde{\Gamma}(c)$, being of the form $g_n(t)=((g_n)_1(t),0)\in E$, $\forall t\in[0,1]$, such that
$$
\max\limits_{t\in[0,1]}\tilde{I}(g_n(t))\in[\gamma_\mu(c)
-\frac{1}{n},\gamma_\mu(c)+\frac{1}{n}].
$$
Then, by Proposition \ref{Pro:three-tildeJ}-\emph{(ii)}, it brings that
$$
|s_n|^2=|s_n-0|^2\leq\min\limits_{t\in[0,1]}\|(v_n,s_n)-g_n(t)\|^2_E\leq\frac{1}{n}.
$$
Consequently, we obtain that
$$
\|(\tilde{h}_n,0)\|^2_{E}=\|\tilde{h}_n\|^2_{H^2}=e^{-4s_n}\int_{\mathbb R^N}|\Delta h_n(x)|^2dx+\int_{\mathbb R^N}|h_n(x)|^2dx\leq e^4\|h_n\|^2_{H^2}.
$$
\end{proof}

Next, we make some further discussions about the sequence $\{u_n\}$ appearing in Lemma \ref{Lem:PS-true}, when the corresponding mountain pass level $\gamma_\mu(c)$ is restricted to one reasonable scope.

\begin{lemma}\label{pro:solution-either-or-true}
Let $N\geq 5$, $\bar{p}\leq p<4^*$ and $c,\mu>0$. In particular, assume that $c\in(0,c_0)$ when $p=\bar{p}$. Suppose that $\{u_n\}\subset S_r(c)$ is the sequence given in Lemma \ref{Lem:PS-true} with
$0<\gamma_\mu(c)<\frac{4+\alpha}{2(N+\alpha)}S^\frac{N+\alpha}{4+\alpha}_\alpha$, where $S_\alpha$ and $c_0$ are given in (\ref{eqn:S-alpha-define}) and (\ref{eqn:defn-c0}), respectively. Then one of the following alternatives holds:
\begin{itemize}
\item[(i)] either up to a subsequence $u_n\rightharpoonup \bar{u}$ in $H^2_{r}({\mathbb R^N})$ but not strongly, where $\bar{u}\neq0$ is a solution to
\begin{equation}\label{eqn:equation-lambdac-solution}
{\Delta}^2u= \lambda u +\mu |u|^{p-2}u+(I_\alpha*|u|^{4^*_\alpha})|u|^{4^*_\alpha-2}u\ \ \mbox{in}\ \ \mathbb{R}^N
\end{equation}
with $\lambda_n\to \bar{\lambda}$ for some $\bar{\lambda}<0$ and $I(\bar{u})\leq\gamma_\mu(c)-\frac{4+\alpha}{2(N+\alpha)}S^\frac{N+\alpha}{4+\alpha}_\alpha$, where $\lambda_n$ is given in
(\ref{eqn:lambda-n-define}) below;
\item[(ii)] or up to a subsequence $u_n\to \bar{u}$ strongly in $H^2_{r}(\mathbb{R}^N)$, $I(\bar{u})=\gamma_\mu(c)$ and $u$ solves problem (\ref{eqn:BS-equation-L2-Super+Critical}) for some $\bar{\lambda}<0$.
\end{itemize}
\end{lemma}
\begin{proof}
For the sequence $\{u_n=\mathcal{H}(v_n,s_n)\}$ determined in Lemma \ref{Lem:PS-true},
making use of \cite[Proposition 5.12]{Willem1996}, we know that there exists $\{\lambda_n\}\subset \mathbb R$ such that
$$
I'(u_n)-\lambda_n\Psi'(u_n)\to 0\ \ \ \mbox{as}\ \ \ n\to\infty,
$$
where $\Psi:H^2_{r}(\mathbb{R}^N)\to \mathbb{R}$ is given by
$$
\Psi(u)=\frac{1}{2}\int_{\mathbb{R}^N}|u|^2dx.
$$
As a consequence, it infers that
\begin{equation}\label{eqn:Jun-varphi}
\int_{\mathbb R^N}[\Delta u_n \Delta \varphi -\lambda_n u_n \varphi-\mu|u_n|^{p-2}u_n\varphi-(I_\alpha*|u_n|^{4^*_\alpha})|u_n|^{4^*_\alpha-2}u_n\varphi] dx=o_n(1)\|\varphi\|, \forall \varphi\in H_r^2({\mathbb R^N}).
\end{equation}
Due to the boundedness of $\{u_n\}$ from Lemma \ref{Lem:PS-true}-\emph{(iii)}, we see that
\begin{equation}\label{eqn:Jun-un-on1}
\|\Delta u_n\|^2_2-\lambda_n\|u_n\|^2_2-\mu\|u_n\|^p_p-\int_{\mathbb R^N}(I_\alpha*|u_n|^{4^*_\alpha})|u_n|^{4^*_\alpha}dx=o_n(1).
\end{equation}
Accordingly, the Lagrange multiplier $\lambda_n$ satisfies
\begin{equation}\label{eqn:lambda-n-define}
\lambda_n=\frac{1}{c}\left\{\|\Delta u_n\|^2_2-\mu\|u_n\|^p_p-\int_{\mathbb R^N}(I_\alpha*|u_n|^{4^*_\alpha})|u_n|^{4^*_\alpha}dx\right\}+o_n(1),
\end{equation}
which indicates that $\{\lambda_n\}$ is also a bounded sequence. Meanwhile, take into account the boundedness of $\{u_n\}$ and the compactness of
$H^2_{r}({\mathbb R^N})\hookrightarrow L^p({\mathbb R^N})$, there exists $\bar{u}\in H^2_{r}({\mathbb R^N})$ such that, up to a subsequence if necessary,
\begin{equation}\label{eqn:un-uc-H2-Lp-RN}
u_n \rightharpoonup \bar{u}\ \ \mbox{in}\ \ H^2_{r}(\mathbb{R}^N),\ \ u_n\to \bar{u} \ \ \mbox{in}\ \ L^p(\mathbb{R}^N)\ \ \mbox{and}\ \ u_n(x)\to \bar{u}(x) \ \ \mbox{a.e. in}\ \ \mathbb{R}^N.
\end{equation}
Let $w_n:=u_n-\bar{u}$, then (\ref{eqn:un-uc-H2-Lp-RN}) states that
\begin{equation}\label{eqn:wn-H2-Lp-RN}
w_n \rightharpoonup 0\ \ \mbox{in}\ \ H^2_{r}(\mathbb{R}^N),\ \ w_n\to 0 \ \ \mbox{in}\ \ L^p(\mathbb{R}^N)\ \ \mbox{and}\ \ w_n(x)\to 0 \ \ \mbox{a.e. in}\ \ \mathbb{R}^N.
\end{equation}

Since $\{\lambda_n\}$ is bounded, we are allowed to choose a subsequence such that $\lambda_n\to \bar{\lambda}\in\mathbb R$. By means of (\ref{eqn:Jun-un-on1}), $\partial_s\tilde{I}(v_n,s_n)\to 0$, $u_n\to \bar{u} \ \ \mbox{in}\ \ L^p(\mathbb{R}^N)$ and $\gamma_p=\frac{N}{2}(\frac{1}{2}-\frac{1}{p})<1$, we deduce that
\begin{equation}\label{eqn:lambdac-leq-0}
\begin{aligned}
\bar{\lambda} c&=\lim\limits_{n\to\infty}\lambda_n\|u_n\|^2_2\\
&=\lim\limits_{n\to\infty}\Big(\|\Delta u_n\|^2_2-\mu\|u_n\|^p_p-\int_{\mathbb R^N}(I_\alpha*|u_n|^{4^*_\alpha})|u_n|^{4^*_\alpha}dx\Big)\\
&=\lim\limits_{n\to\infty}\mu(\gamma_p-1)\|u_n\|^p_p\\
&=\mu(\gamma_p-1)\|\bar{u}\|^p_p\\
&\leq 0.
\end{aligned}
\end{equation}
We next claim that the weak limit $\bar{u}$ does not vanish identically. Suppose by contradiction that $\bar{u}=0$.
Since $\{u_n\}$ is bounded in $H^2(\mathbb R^N)$, up to a subsequence we can assume that $\|\Delta u_n\|^2_2\to l\geq0$. Recalling that $\partial_s\tilde{I}(v_n,s_n)\to 0$
and $u_n\to 0$ in $L^p(\mathbb R^N)$, we have
$$
0\leq\int_{\mathbb R^N}(I_\alpha*|u_n|^{4^*_\alpha})|u_n|^{4^*_\alpha}dx=\|\Delta u_n\|^2_2-\mu\gamma_p\|u_n\|^p_p+o_n(1)\to l.
$$
Therefore, in light of (\ref{eqn:S-alpha-define}), it gives that $l\geq S_\alpha l^\frac{1}{4^*_\alpha}$, which implies that either $l\geq S^\frac{N+\alpha}{4+\alpha}_\alpha$ or $l=0$. If $l\geq S^\frac{N+\alpha}{4+\alpha}_\alpha$, we infer that
$$
\begin{aligned}
\gamma_\mu(c)+o_n(1)=I(u_n)&=\frac{1}{2}\|\Delta u_n\|^2_2-\frac{\mu}{p}\|u_n\|^p_p-\frac{1}{24^*_\alpha}\int_{\mathbb R^N}(I_\alpha*|u_n|^{4^*_\alpha})|u_n|^{4^*_\alpha}dx\\
&=\left(\frac{1}{2}-\frac{1}{24^*_\alpha}\right)\|\Delta u_n\|^2_2-\mu\left(\frac{1}{p}-\frac{\gamma_p}{24^*_\alpha}\right)\|u_n\|^p_p+o_n(1)\\
&=\frac{4+\alpha}{2(N+\alpha)}\|\Delta u_n\|^2_2-\mu\left(\frac{1}{p}-\frac{\gamma_p}{24^*_\alpha}\right)\|u_n\|^p_p+o_n(1)\\
&\to \frac{4+\alpha}{2(N+\alpha)}l,
\end{aligned}
$$
since $I(u_n)\to \gamma_\mu(c)$ and Lemma \ref{Lem:PS-true}-\emph{(ii)}. Hence, one has
$$
\gamma_\mu(c)=\frac{4+\alpha}{2(N+\alpha)}l\geq\frac{4+\alpha}{2(N+\alpha)}S^\frac{N+\alpha}{4+\alpha}_\alpha,
$$
which contradicts with the assumption $\gamma_\mu(c)<\frac{4+\alpha}{2(N+\alpha)}S^\frac{N+\alpha}{4+\alpha}_\alpha$. When $l=0$, we see that
$$
\|u_n\|_p\to 0,\ \ \ \|\Delta u_n\|_2\to 0\ \ \ \mbox{and}\ \ \  \int_{\mathbb R^N}(I_\alpha*|u_n|^{4^*_\alpha})|u_n|^{4^*_\alpha}dx\to 0\ \ \ \mbox{as}\ \ \ n\to\infty,
$$
which brings that $I(u_n)\to 0=\gamma_\mu(c)$, an obvious contradiction. Hence, the weak limit $\bar{u}$ does not vanish identically, and
it follows from (\ref{eqn:lambdac-leq-0}) that $\bar{\lambda}<0$. Passing to the limit in (\ref{eqn:Jun-varphi}) by weak convergence, we see that $\bar{u}$ is a weak solution of
\begin{equation}\label{eqn:equation-1}
\Delta^2u=\bar{\lambda} u+\mu|u|^{p-2}u+(I_\alpha*|u|^{4^*_\alpha})|u|^{4^*_\alpha-2}u\ \ \mbox{in}\ \ H_r^{-2}(\mathbb R^N)
\end{equation}
Moreover, we know that $\bar{u}$ also satisfies the following identify
{\color{red}{\begin{equation}\label{eqn:Pohozaev-identify-baru}
\begin{aligned}
P(\bar{u})&=\int_{\mathbb R^N}|\Delta \bar{u}|^2dx+\mu \gamma_p\int_{\mathbb R^N}|\bar{u}|^pdx+\int_{\mathbb R^N}(I_\alpha*|\bar{u}|^{4^*_\alpha})|\bar{u}|^{4^*_\alpha}dx=0.
\end{aligned}
\end{equation}}}

By virtue of (\ref{eqn:wn-H2-Lp-RN}), the Br\'{e}zis-Lieb Lemma (\!\!\cite[Lemma 1.32]{Willem1996}) and \cite[(3.11)]{Chen-chen2023} ensure that
\begin{equation}\label{eqn:unuo(1)}
\int_{\mathbb{R}^N}|u_n|^2dx=\int_{\mathbb{R}^N}|w_n|^2dx+\int_{\mathbb{R}^N}|\bar{u}|^2dx+o_n(1),
\end{equation}
\begin{equation}\label{eqn:unpupo(1)}
\int_{\mathbb{R}^N}|u_n|^pdx=\int_{\mathbb{R}^N}|w_n|^pdx+\int_{\mathbb{R}^N}|\bar{u}|^pdx+o_n(1)
\end{equation}
and
\begin{equation}\label{eqn:un4-wnp-uc4}
\int_{\mathbb R^N}(I_\alpha*|u_n|^{4^*_\alpha})|u_n|^{4^*_\alpha}dx=\int_{\mathbb R^N}(I_\alpha*|w_n|^{4^*_\alpha})|w_n|^{4^*_\alpha}dx+\int_{\mathbb R^N}(I_\alpha*|\bar{u}|^{4^*_\alpha})|\bar{u}|^{4^*_\alpha}dx+o_n(1).
\end{equation}
In addition, recall that $w_n=u_n-\bar{u}\rightharpoonup 0$ in $H^2_{r}(\mathbb{R}^N)$ and together with (\ref{eqn:unuo(1)}), it also indicates that
\begin{equation}\label{eqn:DunDuo(1)}
\int_{\mathbb{R}^N}|\Delta u_n|^2dx=\int_{\mathbb{R}^N}|\Delta w_n|^2dx+\int_{\mathbb{R}^N}|\Delta \bar{u}|^2dx+o_n(1).
\end{equation}
Therefore, based on Lemma \ref{Lem:PS-true}-\emph{(ii)} and $u_n\to \bar{u}$ in $L^p(\mathbb R^N)$, we deduce from (\ref{eqn:un4-wnp-uc4}) and
(\ref{eqn:DunDuo(1)}) that
$$
\begin{aligned}
\|\Delta \bar{u}\|^2_2+\|\Delta w_n\|^2_2&=\|\Delta u_n\|^2_2+o_n(1)\\
&=\mu \gamma_p\|u_n\|^p_p+\int_{\mathbb R^N}(I_\alpha*|u_n|^{4^*_\alpha})|u_n|^{4^*_\alpha}dx+o_n(1)\\
&=\mu \gamma_p\|\bar{u}\|^p_p+\int_{\mathbb R^N}(I_\alpha*|w_n|^{4^*_\alpha})|w_n|^{4^*_\alpha}dx+\int_{\mathbb R^N}(I_\alpha*|\bar{u}|^{4^*_\alpha})|\bar{u}|^{4^*_\alpha}dx+o_n(1).
\end{aligned}
$$
Furthermore, according to (\ref{eqn:Pohozaev-identify-baru}), there holds that
$$
\|\Delta w_n\|^2_2=\int_{\mathbb R^N}(I_\alpha*|w_n|^{4^*_\alpha})|w_n|^{4^*_\alpha}dx+o_n(1).
$$
So, up to a subsequence, it can be assumed that
\begin{equation}\label{eqn:Deltawn-wn4*-l}
\lim\limits_{n\to\infty}\|\Delta w_n\|^2_2=\lim\limits_{n\to\infty}\int_{\mathbb R^N}(I_\alpha*|w_n|^{4^*_\alpha})|w_n|^{4^*_\alpha}dx=l\geq 0.
\end{equation}
Using (\ref{eqn:S-alpha-define}) again, we have $l\geq S_\alpha l^\frac{1}{4^*_\alpha}$, that is, either $l=0$ or $l\geq S^\frac{N+\alpha}{4+\alpha}_\alpha$.

If $l\geq S^\frac{N+\alpha}{4+\alpha}_\alpha$, (\ref{eqn:unpupo(1)})-(\ref{eqn:Deltawn-wn4*-l}) guarantee that
$$
\begin{aligned}
\gamma_\mu(c)=\lim\limits_{n\to\infty}I(u_n)&=\lim\limits_{n\to\infty}\Big(I(\bar{u})+\frac{1}{2}\|\Delta w_n\|^2_2-\frac{1}{24^*_\alpha}\int_{\mathbb R^N}(I_\alpha*|w_n|^{4^*_\alpha})|w_n|^{4^*_\alpha}dx\Big)\\
&=I(\bar{u})+\Big(\frac{1}{2}-\frac{1}{24^*_\alpha}\Big)l\\
&=I(\bar{u})+\frac{4+\alpha}{2(N+\alpha)}l\\
&\geq I(\bar{u})+\frac{4+\alpha}{2(N+\alpha)}S^\frac{N+\alpha}{4+\alpha}_\alpha,
\end{aligned}
$$
whence alternative $(i)$ in the thesis of this lemma follows. If instead $l=0$, we need to show that $u_n\to \bar{u}$ in $H^2_r(\mathbb R^N)$. From (\ref{eqn:Deltawn-wn4*-l}), we have known that
\begin{equation}\label{eqn:wn-un-2-4*}
\lim\limits_{n\to\infty}\|\Delta w_n\|^2_2=\lim\limits_{n\to\infty}\int_{\mathbb R^N}(I_\alpha*|w_n|^{4^*_\alpha})|w_n|^{4^*_\alpha}dx=0.
\end{equation}
Hence, it remains to verify that $u_n\to \bar{u}$ in $L^2(\mathbb R^N)$. To this end, test (\ref{eqn:Jun-varphi}) and (\ref{eqn:equation-1}) with $w_n=u_n-\bar{u}$, it yields that
\begin{equation}\label{eqn:un-uc-lambda=Lp-4}
\begin{aligned}
\int_{\mathbb R^N}&[|\Delta(u_n-\bar{u})|^2-(\lambda_n u_n-\bar{\lambda} \bar{u})(u_n-\bar{u})]dx\\
&=\mu\int_{\mathbb R^N}(|u_n|^{p-2}u_n-|\bar{u}|^{p-2}\bar{u})(u_n-\bar{u})dx\\
&\ \ \ +\int_{\mathbb R^N}\Big((I_\alpha*|u_n|^{4^*_\alpha})|u_n|^{4^*_\alpha-2}u_n-(I_\alpha*|\bar{u}|^{4^*_\alpha})|\bar{u}|^{4^*_\alpha-2}\bar{u}\Big)(u_n-\bar{u})dx+o_n(1).
\end{aligned}
\end{equation}
Observing that $u_n\to \bar{u}$ in $L^p(\mathbb R^N)$ from (\ref{eqn:un-uc-H2-Lp-RN}), we derive that
\begin{equation}\label{eqn:un-p2-un-uc-0}
%\begin{aligned}
\left|\int_{\mathbb R^N}(|u_n|^{p-2}u_n-|\bar{u}|^{p-2}\bar{u})(u_n-\bar{u})dx\right|\leq\||u_n|^{p-2}u_n-|\bar{u}|^{p-2}\bar{u}\|_\frac{p}{p-1}\|u_n-\bar{u}\|_p
\to 0.
\end{equation}
Additionally, in view of \cite[(4.24)]{Chen-chen2023}, we know that
\begin{equation}\label{eqn:I-alpha-to-0}
\int_{\mathbb R^N}\Bigl[(I_\alpha*|u_n|^{4^*_\alpha})|u_n|^{4^*_\alpha-2}u_n-(I_\alpha*|\bar{u}|^{4^*_\alpha})|\bar{u}|^{4^*_\alpha-2}\bar{u}\Bigr](u_n-\bar{u})dx\to 0.
\end{equation}
Therefore, from (\ref{eqn:un-uc-lambda=Lp-4})-(\ref{eqn:I-alpha-to-0}), it ensures that
$$
\begin{aligned}
0&=\lim\limits_{n\to\infty}\int_{\mathbb R^N}(\lambda_n u_n-\bar{\lambda} \bar{u})(u_n-\bar{u})dx\\
&=\lim\limits_{n\to\infty}\int_{\mathbb R^N}(\lambda_n u_n-\bar{\lambda} u_n+\bar{\lambda} u_n-\bar{\lambda} \bar{u})(u_n-\bar{u})dx\\
&=\bar{\lambda}\lim\limits_{n\to\infty}\|u_n-\bar{u}\|_2^2.
\end{aligned}
$$
At this point, this lemma is proved.
\end{proof}

\section{Proof of Theorem \ref{Thm:normalized-bScritical-solutions-L2super+L2critical}}\label{sec:proof-main-theorem-L2critical+L2super}
Based on the preliminaries in Section \ref{sec:preliminary}, it is sufficient to estimate the mountain pass level $\gamma_{\mu.c}\in \Bigl(0,\frac{4+\alpha}{2(N+\alpha)}S^\frac{N+\alpha}{4+\alpha}_\alpha\Bigr)$ to obtain the existence of normalized solution of problem (\ref{eqn:BS-equation-L2-Super+Critical}). For this purpose, we firstly introduce two extremal functions frequently used in the sequel arguments.
\begin{lemma}(\!\!\cite[Page 16]{Chen-chen2023})\label{lem:U-bi-equation}
For any $\varepsilon>0$, consider the following extremal function
$$
U_\varepsilon(x):=\frac{[N(N+2)(N-2)(N-4) \varepsilon^4]^\frac{N-4}{8}}{(\varepsilon^2+|x|^2)^\frac{N-4}{2}}, \ \ x \in \mathbb{R}^N.
$$
Define
\begin{equation}\label{eqn:Defn-S}
S:=\inf_{u\in H^2(\mathbb R^N)\setminus\{0\}}\frac{\int_{\mathbb R^N}|\Delta u|^2dx}{(\int_{\mathbb R^N}|u|^{4^*}dx)^\frac{2}{4^*}},
\end{equation}
then, the minimizers of $S$ can be obtained by $U_\varepsilon(x)$. Moreover, $U_\varepsilon(x)$ solves the following critical equation
$$
\Delta^2 u=|u|^{4^*-2}u\ \ \mbox{in}\ \  \mathbb R^N.
$$
\end{lemma}

\begin{lemma}(\!\!\cite[Page 228]{Rani-Goyal2022})\label{lem:U-bi-equation}
Let
$$
\widetilde{U}_{\varepsilon}(x):=S^\frac{\alpha(4-N)}{8(4+\alpha)} (C(N, \alpha))^\frac{4-N}{2(4+\alpha)} U_{\varepsilon}(x),\ \ \forall \varepsilon>0,
$$
where $C(N,\alpha)$ is defined in (\ref{eqn:Defn-CNalpha}). Then, $\widetilde{U}_{\varepsilon}$ gives a family of minimizers for
$$
S_{H, L}:=\inf\limits_{u\in H^2(\mathbb R^N)\backslash\{0\}}\frac{\int_{\mathbb R^N}|\Delta u|^2dx}{\left(\int_{\mathbb R^N}\int_{\mathbb R^N}\frac{|u(x)|^{4^*_\alpha}|u(y)|^{4^*_\alpha}}{|x-y|^{N-\alpha}}dxdy\right)^\frac{1}{4^*_\alpha}}
$$
and satisfies the equation
$$
\Delta^2 u=\Big(\int_{\mathbb{R}^N} \frac{|u(y)|^{4^*_\alpha}}{|x-y|^{N-\alpha}}dy\Big)|u|^{4_\alpha^*-2} u \ \ \mbox{in}\ \  \mathbb{R}^N.
$$
Moreover,
\begin{equation}\label{eqn:widetilde-U}
\int_{\mathbb{R}^N}|\Delta \widetilde{U}_{\varepsilon}|^2 dx=\int_{\mathbb{R}^N} \int_{\mathbb{R}^N} \frac{|\widetilde{U}_{\varepsilon}(x)|^{4_\alpha^*}|\widetilde{U}_{\varepsilon}(y)|^{4_\alpha^*}}{|x-y|^{N-\alpha}} dxdy=\Big(S_{H, L}\Big)^\frac{N+\alpha}{4+\alpha}=\Big(A_\alpha^\frac{1}{4^*_\alpha} S_\alpha\Big)^\frac{N+\alpha}{4+\alpha}.
\end{equation}
Here, $A_\alpha$ and $S_\alpha$ are given in (\ref{eqn:Defn-I-A-alpha}) and (\ref{eqn:S-alpha-define}) respectively.
\end{lemma}

\begin{lemma}\label{lem:bih-eatumate-energy}
Let $B_r(x)$ be the ball of radius of $r$ at $x$ and $\psi \in \mathcal{C}_0^{\infty}(\mathbb{R}^N)$ is a radial cut-off function satisfying:
\begin{itemize}
\item[(i)]  $0 \leq \psi(x) \leq 1$ for any $x\in\mathbb R^N$;
\item[(ii)] $\psi(x) \equiv 1$ in $B_1(0)$;
\item[(iii)]$\psi(x) \equiv 0$ in $\mathbb R^N\setminus B_2(0)$.
\end{itemize}
Define
\begin{equation}\label{eqn:u-varepsilon-define}
u_{\varepsilon}(x):=\psi(x) U_{\varepsilon}(x),
\end{equation}
then the following hold true:
\begin{equation}\label{eqn:delta-eatimate}
\int_{\mathbb{R}^N}|\Delta u_\varepsilon|^2 d x=S^\frac{N}{4}+O(\varepsilon^{N-4}),
\end{equation}
\begin{equation}\label{eqn:4*-eatimate}
\int_{\mathbb{R}^N}| u_\varepsilon|^{4^*} d x=S^\frac{N}{4}+O(\varepsilon^{N})
\end{equation}
and
\begin{equation}\label{eqn:p-estimate}
\int_{\mathbb{R}^N}|u_\varepsilon|^p dx\begin{cases}=K \varepsilon^{N-\frac{(N-4)p}{2}}+o\Big(\varepsilon^{N-\frac{(N-4)p}{2}}\Big), &\,\, \mbox{if}\,\, p > \frac{N}{N-4};\\[0.25cm]
\geq K \varepsilon^\frac{N}{2}|ln \varepsilon|+O\Big(\varepsilon^\frac{N}{2}\Big), &\,\, \mbox{if}\,\,  p = \frac{N}{N-4}; \\[0.25cm]
=K \varepsilon^\frac{(N-4)p}{2}+o\Big(\varepsilon^{\frac{(N-4)p}{2}}\Big), & \,\,\mbox{if}\,\, p < \frac{N}{N-4}
\end{cases}
\end{equation}
for some constant $K>0$, where $S$ is given in (\ref{eqn:Defn-S}).
\end{lemma}
\begin{proof}
Some direct calculations state that
\begin{equation}\label{eqn:Delta-u-varepsilon}
\Delta u_{\varepsilon}=\Delta U_{\varepsilon}\psi+U_{\varepsilon}\Delta\psi+2\nabla U_{\varepsilon}\cdot \nabla\psi,
\end{equation}
\begin{equation}\label{eqn:U-varepsilon-AN}
\nabla U_{\varepsilon}=D_N\varepsilon^{\frac{N-4}{2}}(4-N)\frac{x}{(\varepsilon^2+|x|^2)^{\frac{N-2}{2}}},\
\end{equation}
and
\begin{equation}\label{eqn:Delta-U-varepsilon-AN}
\Delta U_{\varepsilon}=D_N\varepsilon^{\frac{N-4}{2}}(4-N)\left[\frac{1}{(\varepsilon^2+|x|^2)^{\frac{N-2}{2}}}+(2-N)\frac{|x|^2}{(\varepsilon^2+|x|^2)^{\frac{N}{2}}}  \right],
\end{equation}
where
$$
D_N:=[N(N+2)(N-2)(N-4)]^\frac{N-4}{8}.
$$
Then, we deduce that
\begin{equation}\label{eqn:Delta-U-varepsilon-2-2}
\begin{aligned}
\|\Delta U_{\varepsilon}\|^2_2&=\int_{\mathbb{R}^N}|\Delta U_{\varepsilon}|^2dx\\
&=D_N^2\varepsilon^{N-4}(4-N)^2\int_{\mathbb{R}^N}\frac{1}{(\varepsilon^2+|x|^2)^{N-2}}dx\\
&\ \ \  +D_N^2\varepsilon^{N-4}(4-N)^2(2-N)^2\int_{\mathbb{R}^N} \frac{|x|^4}{(\varepsilon^2+|x|^2)^{N}}dx\\
&\ \ \  +2D_N^2\varepsilon^{N-4}(4-N)^2(2-N)\int_{\mathbb{R}^N}\frac{|x|^2}{(\varepsilon^2+|x|^2)^{N-1}} dx   \\
&=D_N^2\varepsilon^{N-4}(4-N)^2\omega\int^{\infty}_0\frac{r^{N-1}}{(\varepsilon^2+r^2)^{N-2}}dr\\
&\ \ \ +D_N^2\varepsilon^{N-4}(4-N)^2(2-N)^2\omega\int^{\infty}_0 \frac{r^{N+3}}{(\varepsilon^2+r^2)^{N}}dr\\
&\ \ \  -2D_N^2\varepsilon^{N-4}(4-N)^2(N-2)\omega\int^{\infty}_0\frac{r^{N+1}}{(\varepsilon^2+r^2)^{N-1}} dr   \\
&=D_N^2\varepsilon^{N-4}(4-N)^2\frac{1}{\varepsilon^{2N-4}}\varepsilon^{N-1}\varepsilon\omega\int^{\infty}_0\frac{r^{N-1}}{(1+r^2)^{N-2}}dr\\
&\ \ \ +D_N^2\varepsilon^{N-4}(4-N)^2\frac{1}{\varepsilon^{2N}}\varepsilon^{N+3}\varepsilon(2-N)^2\omega\int^{\infty}_0\frac{r^{N+3}}{(1+r^2)^N}dr\\
&\ \ \  -2D_N^2\varepsilon^{N-4}(4-N)^2(N-2)\varepsilon^{N+1}\frac{1}{\varepsilon^{2N-2}}\varepsilon\omega\int^{\infty}_0\frac{r^{N+1}}{(1+r^2)^{N-1}}dr\\
&=D_N^2(4-N)^2\omega\int^{\infty}_0\frac{r^{N-1}}{(1+r^2)^{N-2}}+(2-N)^2\frac{r^{N+3}}{(1+r^2)^N}-2(N-2)\frac{r^{N+1}}{(1+r^2)^{N-1}}dr
\end{aligned}
\end{equation}
and
\begin{equation}\label{eqn:Delta-u-varepsilon-2-2}
\begin{aligned}
\|\Delta u_{\varepsilon}\|^2_2&=\int_{\mathbb{R}^N}\Big|\Delta U_{\varepsilon}\psi+U_{\varepsilon}\Delta\psi+2\nabla U_{\varepsilon}\cdot \nabla\psi \Big|^2dx\\
&=\int_{\mathbb{R}^N}|\Delta U_{\varepsilon}\psi|^2+\Delta U_{\varepsilon}\psi U_{\varepsilon}\Delta \psi+2\Delta U_{\varepsilon}\psi \nabla U_{\varepsilon}\cdot\nabla\psi dx\\
&\ \ \ +\int_{\mathbb{R}^N}U_{\varepsilon}\Delta \psi \Delta U_{\varepsilon}\psi+|U_{\varepsilon}\Delta \psi|^2+2U_{\varepsilon}\Delta\psi\nabla U_{\varepsilon}\cdot\nabla\phi dx\\
&\ \ \ +\int_{\mathbb{R}^N}2\nabla U_{\varepsilon}\cdot\nabla\psi\Delta U_{\varepsilon}\psi+2\nabla U_{\varepsilon}\cdot\nabla\psi U_{\varepsilon}\Delta \psi+4|\nabla U_{\varepsilon}\cdot\nabla\psi|^2dx\\
&=2A_1(\varepsilon)+A_2(\varepsilon)+A_3(\varepsilon)+A_4(\varepsilon),
\end{aligned}
\end{equation}
where $\omega$ is the area of the unit sphere in $\mathbb{R}^N$,
$$
A_1(\varepsilon):=\int_{\mathbb{R}^N}\Delta U_{\varepsilon}\psi U_{\varepsilon}\Delta \psi+2\Delta U_{\varepsilon}\psi  \nabla U_{\varepsilon}\cdot\nabla\psi+2U_{\varepsilon}\Delta\psi \nabla U_{\varepsilon}\cdot\nabla\psi dx,
$$
$$
A_2(\varepsilon):=\int_{\mathbb{R}^N}|\Delta U_{\varepsilon}\psi|^2dx,
$$
$$
A_3(\varepsilon):=\int_{\mathbb{R}^N}|U_\varepsilon\Delta\psi|^2dx
$$
and
$$
A_4(\varepsilon):=4\int_{\mathbb{R}^N}|\nabla U_{\varepsilon}\cdot\nabla \psi|^2dx.
$$

Observe the existence of $\widetilde{K}$ such that $|\nabla\psi|\le \widetilde{K}$ and $|\Delta \psi|\le \widetilde{K}$, it immediately brings that
\begin{equation}\label{eqn:A1-1}
\begin{aligned}
&\int_{\mathbb{R}^N}|\Delta U_{\varepsilon}\psi U_{\varepsilon}\Delta \psi|dx\\
&=\int_{B_1(0)}|\Delta U_{\varepsilon}\psi U_{\varepsilon}\Delta \psi|dx+\int_{B_2(0)\backslash B_1(0)}\!\!\!\!\!\!\!\!\!\!|\Delta U_{\varepsilon}\psi U_{\varepsilon}\Delta \psi|dx+\int_{\mathbb{R}^N\backslash B_2(0)}\!\!\!\!\!\!\!\!\!\!|\Delta U_{\varepsilon}\psi U_{\varepsilon}\Delta \psi|dx\\
&=\int_{B_2(0)\backslash B_1(0)}|\Delta U_{\varepsilon}\psi U_{\varepsilon}\Delta \psi|dx\\
&=\int_{B_2(0)\backslash B_1(0)}D_N^2\varepsilon^{N-4}(N-4)\left|\frac{\psi}{(\varepsilon^2+|x|^2)^{\frac{N-2}{2}}}+(2-N)\frac{|x|^2\psi}{(\varepsilon^2+|x|^2)^{\frac{N}{2}}}  \right|\frac{|\Delta \psi|}{(\varepsilon^2+|x|^2)^{\frac{N-4}{2}}} dx\\
&\leq D_N^2\varepsilon^{N-4}(N-4)\widetilde{K}^2\int_{B_2(0)\backslash B_1(0)}\frac{dx}{(\varepsilon^2+|x|^2)^{N-3}}\\
&\ \ \ +D_N^2\varepsilon^{N-4}(N-4)(N-2)\widetilde{K}\int_{B_2(0)\backslash B_1(0)}\frac{|x|^2}{(\varepsilon^2+|x|^2)^{N-2}}dx\\
&= D_N^2\varepsilon^{N-4}(N-4)\omega\widetilde{K}\int_1^2\frac{r^{N-1}}{(\varepsilon^2+r^2)^{N-3}}dr
+D_N^2\varepsilon^{N-4}(N-4)(N-2)\omega\widetilde{K}^2\int_1^2\frac{r^{N+1}}{(\varepsilon^2+r^2)^{N-2}}dr\\
&\leq C \varepsilon^{N-4},
\end{aligned}
\end{equation}
\begin{equation}\label{eqn:A1-2}
\begin{aligned}
&\int_{\mathbb{R}^N}|2\Delta U_{\varepsilon}\psi\nabla U_{\varepsilon}\cdot\nabla\psi|dx\\
&=\int_{B_2(0)\backslash B_1(0)}|2\Delta U_{\varepsilon}\psi\nabla U_{\varepsilon}\cdot\nabla\psi|dx\\
&=\int_{B_2(0)\backslash B_1(0)}\left|2D_N^2\varepsilon^{N-4}(4-N)^2\Big(\frac{\psi}{(\varepsilon^2+|x|^2)^{\frac{N-2}{2}}}+(N-2)\frac{|x|^2\psi}{(\varepsilon^2+|x|^2)^{\frac{N}{2}}}  \Big)\frac{x\cdot\nabla\psi}{(\varepsilon^2+|x|^2)^{\frac{N-2}{2}}}\right|dx\\
&\leq2D_N^2\varepsilon^{N-4}(4-N)^2\widetilde{K}\int_{B_2(0)\backslash B_1(0)}\frac{|x|}{(\varepsilon^2+|x|^2)^{N-2}}dx\\
&\ \ \ +2D_N^2\varepsilon^{N-4}(4-N)^2(N-2)\widetilde{K}\int_{B_2(0)\backslash B_1(0)}\frac{|x|^3}{(\varepsilon^2+|x|^2)^{N-1}}dx\\
&=2D_N^2\varepsilon^{N-4}(4-N)^2\omega\widetilde{K}\int_1^2\frac{r^{N}}{(\varepsilon^2+r^2)^{N-2}}dr +2D_N^2\varepsilon^{N-4}(4-N)^2(N-2)\omega\widetilde{K}\int_1^2\frac{r^{N+2}}{(\varepsilon^2+r^2)^{N-1}}dr\\
&\leq C\varepsilon^{N-4}
\end{aligned}
\end{equation}
and
\begin{equation}\label{eqn:A1-3}
\begin{aligned}
&\int_{\mathbb{R}^N}|2U_{\varepsilon}\Delta\psi \nabla U_{\varepsilon}\cdot\nabla\psi|dx\\
&=\int_{B_1(0)}|2U_{\varepsilon}\Delta\psi \nabla U_{\varepsilon}\cdot\nabla\psi|dx+\int_{B_2(0)\backslash B_1(0)}|2U_{\varepsilon}\Delta\psi \nabla U_{\varepsilon}\cdot\nabla\psi|dx\\
&\ \ \ +\int_{\mathbb{R}^N\backslash B_2(0)}|2U_{\varepsilon}\Delta\psi \nabla U_{\varepsilon}\cdot\nabla\psi|dx\\
&=\int_{B_2(0)\backslash B_1(0)}|2U_{\varepsilon}\Delta\psi \nabla U_{\varepsilon}\cdot\nabla\psi|dx\\
&=\int_{B_2(0)\backslash B_1(0)}\left|2D_N^2(4-N)\varepsilon^{N-4}\frac{\Delta \psi x\cdot\nabla\psi}{(\varepsilon^2+|x|^2)^{\frac{N-4}{2}+\frac{N-2}{2}}}\right|dx\\
&\leq2D_N^2(N-4)\varepsilon^{N-4}\widetilde{K}^2\int_{B_2(0)\backslash B_1(0)}\frac{ |x|}{(\varepsilon^2+|x|^2)^{\frac{N-4}{2}+\frac{N-2}{2}}}dx\\
&= 2D_N^2(N-4)\varepsilon^{N-4}\omega\int_1^2\frac{r^N}{(\varepsilon^2+r^2)^{N-3}}dr\\
&\leq C\varepsilon^{N-4}.
\end{aligned}
\end{equation}
Thus, (\ref{eqn:A1-1})-(\ref{eqn:A1-3}) imply that
\begin{equation}\label{eqn:A1-estimate}
A_1(\varepsilon)=O(\varepsilon^{N-4}).
\end{equation}

In order to estimate $A_2(\varepsilon)$, we need the subsequent facts:
\begin{equation}\label{eqn:A2-01}
\begin{aligned}
&D_N^2\varepsilon^{N-4}(4-N)^2\int_{\mathbb{R}^N}\frac{\psi^2}{(\varepsilon^2+|x|^2)^{N-2}}dx\\
&=D_N^2\varepsilon^{N-4}(4-N)^2\int_{B_2(0)}\frac{\psi^2}{(\varepsilon^2+|x|^2)^{N-2}}dx\\
&=D_N^2\varepsilon^{N-4}(4-N)^2\omega\int_0^2\frac{r^{N-1}\psi^2}{(\varepsilon^2+r^2)^{N-2}}dr\\
&=D_N^2\varepsilon^{N-4}(4-N)^2\omega\int_0^2\Big(\frac{r^{N-1}(\psi^2-1)}{(\varepsilon^2+r^2)^{N-2}}+\frac{r^{N-1}}{(\varepsilon^2+r^2)^{N-2}}\Big)dr\\
&=D_N^2\varepsilon^{N-4}(4-N)^2\omega\int_0^1\frac{r^{N-1}(\psi^2-1)}{(\varepsilon^2+r^2)^{N-2}}dr+D_N^2\varepsilon^{N-4}(4-N)^2\omega
\int_1^2\frac{r^{N-1}(\psi^2-1)}{(\varepsilon^2+r^2)^{N-2}}dr\\
&\ \ \ +D_N^2\varepsilon^{N-4}(4-N)^2\omega\int_0^2\frac{r^{N-1}}{(\varepsilon^2+r^2)^{N-2}}dr\\
&=D_N^2\varepsilon^{N-4}(4-N)^2\omega\int_0^2\frac{r^{N-1}}{(\varepsilon^2+r^2)^{N-2}}dr+O(\varepsilon^{N-4})\\
&=D_N^2(4-N)^2\omega\int_0^{\frac{2}{\varepsilon}}\frac{r^{N-1}}{(1+r^2)^{N-2}}dr+O(\varepsilon^{N-4}),
\end{aligned}
\end{equation}
\begin{equation}\label{eqn:A2-02}
\begin{aligned}
&D_N^2\varepsilon^{N-4}(4-N)^2(2-N)^2 \int_{\mathbb{R}^N}\frac{|x|^4\psi^2}{(\varepsilon^2+|x|^2)^{N}}dx\\
&=D_N^2\varepsilon^{N-4}(4-N)^2(2-N)^2 \int_{B_2(0)}\frac{|x|^4\psi^2}{(\varepsilon^2+|x|^2)^{N}}dx\\
&=D_N^2\varepsilon^{N-4}(4-N)^2(2-N)^2\omega\int_0^2\frac{r^{N+3}\psi^2}{(\varepsilon^2+r^2)^{N}}dr\\
&=D_N^2\varepsilon^{N-4}(4-N)^2(2-N)^2\omega\int_0^2\Big(\frac{r^{N+3}(\psi^2-1)}{(\varepsilon^2+r^2)^{N}}+\frac{r^{N+3}}{(\varepsilon^2+r^2)^{N}} \Big)dr\\
&=D_N^2\varepsilon^{N-4}(4-N)^2(2-N)^2\omega\int_0^1\frac{r^{N+3}(\psi^2-1)}{(\varepsilon^2+r^2)^{N}}dr\\
&\ \ \ +D_N^2\varepsilon^{N-4}(4-N)^2(2-N)^2\omega\int_1^2\frac{r^{N+3}(\psi^2-1)}{(\varepsilon^2+r^2)^{N}}dr\\
&\ \ \ +D_N^2\varepsilon^{N-4}(4-N)^2(2-N)^2\omega\int_0^2\frac{r^{N+3}}{(\varepsilon^2+r^2)^{N}} dr\\
&=D_N^2(4-N)^2(2-N)^2\omega\int_0^{\frac{2}{\varepsilon}}\frac{r^{N+3}}{(1+r^2)^{N}}dr+O(\varepsilon^{N-4})
\end{aligned}
\end{equation}
and
\begin{equation}\label{eqn:A2-03}
\begin{aligned}
&2D_N^2\varepsilon^{N-4}(4-N)^2(N-2)\int_{\mathbb{R}^N}\frac{|x|^2\psi^2}{(\varepsilon^2+|x|^2)^{N-1}}dx\\
&=2D_N^2\varepsilon^{N-4}(4-N)^2(N-2)\omega\int_0^2\frac{r^{N+1}\psi^2}{(\varepsilon^2+r^2)^{N-1}}dr\\
&=2D_N^2\varepsilon^{N-4}(4-N)^2(N-2)\omega\int_0^2\Big(\frac{r^{N+1}(\psi^2-1)}{(\varepsilon^2+r^2)^{N-1}}+\frac{r^{N+1}}{(\varepsilon^2+r^2)^{N-1}} \Big)dr\\
&=2D_N^2\varepsilon^{N-4}(4-N)^2(N-2)\omega\int_0^1\frac{r^{N+1}(\psi^2-1)}{(\varepsilon^2+r^2)^{N-1}}dr\\
&\ \ \ +2D_N^2\varepsilon^{N-4}(4-N)^2(N-2)\omega\int_1^2\frac{r^{N+1}(\psi^2-1)}{(\varepsilon^2+r^2)^{N-1}}dr\\
&\ \ \ +2D_N^2\varepsilon^{N-4}(4-N)^2(N-2)\omega\int_0^2\frac{r^{N+1}}{(\varepsilon^2+r^2)^{N-1}}dr\\
&=2D_N^2(4-N)^2(N-2)\omega\int_0^{\frac{2}{\varepsilon}}\frac{r^{N+1}}{(1+r^2)^{N-1}}dr+O(\varepsilon^{N-4}).
\end{aligned}
\end{equation}
Based on (\ref{eqn:A2-01})-(\ref{eqn:A2-03}), there holds that
\begin{equation}\label{eqn:A2-estimate}
\begin{aligned}
A_2(\varepsilon)&=\int_{\mathbb{R}^N}|\Delta U_{\varepsilon}\psi|^2dx\\
&=\int_{\mathbb{R}^N}D_N^2\varepsilon^{N-4}(4-N)^2\left(\frac{\psi}{(\varepsilon^2+|x|^2)^{\frac{N-2}{2}}}+(2-N)\frac{|x|^2\psi}{(\varepsilon^2+|x|^2\psi)^{\frac{N}{2}}}  \right)^2dx\\
&=D_N^2\varepsilon^{N-4}(4-N)^2\!\!\int_{\mathbb{R}^N}\!\!\frac{\psi^2}{(\varepsilon^2+|x|^2)^{N-2}}dx\!+\!D_N^2\varepsilon^{N-4}(4-N)^2(2-N)^2\!\!\int_{\mathbb{R}^N}
\!\!\frac{|x|^4\psi^2}{(\varepsilon^2+|x|^2)^{N}}dx\\
&\ \ \ -2D_N^2\varepsilon^{N-4}(4-N)^2(N-2)\int_{\mathbb{R}^N}\frac{|x|^2|\psi|^2}{(\varepsilon^2+|x|^2)^{N-1}}dx\\
&=D_N^2(4-N)^2\omega\int_0^{\frac{2}{\varepsilon}}\frac{r^{N-1}}{(1+r^2)^{N-2}}dr
+D_N^2(4-N)^2(2-N)^2\omega\int_0^{\frac{2}{\varepsilon}}\frac{r^{N+3}}{(1+r^2)^{N}}dr\\
&\ \ \ -2D_N^2(4-N)^2(N-2)\omega\int_0^{\frac{2}{\varepsilon}}\frac{r^{N+1}}{(1+r^2)^{N-1}}dr+O(\varepsilon^{N-4}).
\end{aligned}
\end{equation}

With regard to $A_3(\varepsilon)$ and $A_4(\varepsilon)$, obviously one has
\begin{equation}\label{eqn:A3-estimate}
\begin{aligned}
A_3(\varepsilon)&=\int_{\mathbb{R}^N}|U_\varepsilon\Delta\psi|^2dx\\
&=\int_{B_2(0)\backslash B_1(0)}|U_\varepsilon\Delta\psi|^2dx\\
&=\int_{B_2(0)\backslash B_1(0)}D_N^2\varepsilon^{N-4}\frac{|\Delta\psi|^2}{(\varepsilon^2+|x|^2)^{N-4}}dx\\
&\leq D_N^2\varepsilon^{N-4}\omega\widetilde{K}^2\int_1^2\frac{r^{N-1}}{(\varepsilon^2+r^2)^{N-4}}dr\\
&\leq O(\varepsilon^{N-4})
\end{aligned}
\end{equation}
and
\begin{equation}\label{eqn:A4-estimate}
\begin{aligned}
A_4(\varepsilon)&=4\int_{\mathbb{R}^N}|\nabla U_{\varepsilon}\cdot\nabla \psi|^2dx\\
&=4D_N^2\varepsilon^{N-4}(4-N)^2\int_{\mathbb{R}^N}\frac{|x\cdot\nabla\psi|^2}{(\varepsilon^2+|x|^2)^{N-2}}dx\\
&\leq4D_N^2\varepsilon^{N-4}(4-N)^2\int_{\mathbb{R}^N}\frac{|x|^2|\nabla\psi|^2}{(\varepsilon^2+|x|^2)^{N-2}}dx\\
&= 4D_N^2\varepsilon^{N-4}(4-N)^2\omega\widetilde{K}^2\int_1^2\frac{r^{N+1}}{(\varepsilon^2+r^2)^{N-2}}dr\\
&\leq O(\varepsilon^{N-4}).
\end{aligned}
\end{equation}

Together with (\ref{eqn:Delta-U-varepsilon-2-2})-(\ref{eqn:Delta-u-varepsilon-2-2}), (\ref{eqn:A1-estimate}), (\ref{eqn:A2-estimate})-(\ref{eqn:A4-estimate}) and the following observations
$$
\int_{\frac{2}{\varepsilon}}^{\infty}\frac{r^{N-1}}{(1+r^2)^{N-2}}dr\le \int_{\frac{2}{\varepsilon}}^{\infty}r^{3-N}dr=O(\varepsilon^{N-4}),
$$
$$
\int_{\frac{2}{\varepsilon}}^{\infty}\frac{r^{N+3}}{(1+r^2)^{N}}dr\le \int_{\frac{2}{\varepsilon}}^{\infty}r^{3-N}dr=O(\varepsilon^{N-4}),
$$
$$
\int_{\frac{2}{\varepsilon}}^{\infty}\frac{r^{N+1}}{(1+r^2)^{N-1}}dr\le \int_{\frac{2}{\varepsilon}}^{\infty}r^{3-N}dr=O(\varepsilon^{N-4}),
$$
we infer that
$$
\|\Delta u_{\varepsilon}\|^2_2=\|\Delta U_{\varepsilon}\|^2_2+O(\varepsilon^{N-4}).
$$
In addition, from Lemma \ref{lem:U-bi-equation}, we know that $S\|U_\varepsilon\|^2_{4^*}=\|\Delta U_\varepsilon\|^2_2$ and
$\|\Delta U_\varepsilon\|^2_2=\|U_\varepsilon\|^{4^*}_{4^*}$. As a result, we conclude that
\begin{equation}\label{eqn:Delta-u-Delta-u-O-S}
\|\Delta u_{\varepsilon}\|^2_2=\|\Delta U_{\varepsilon}\|^2_2+O(\varepsilon^{N-4})=S^{\frac{N}{4}}+O(\varepsilon^{N-4}).
\end{equation}

Next, we turn to (\ref{eqn:4*-eatimate}). From the expression of $u_\varepsilon$, it gives that
$$
\begin{aligned}
\|u_{\varepsilon}\|^{4^*}_{4^*}&=\int_{\mathbb{R}^N}|\psi U_{\varepsilon}|^{4^*}dx\\
&=D_N^{4^*}\varepsilon^N\int_{\mathbb{R}^N}\frac{\psi^{4^*}}{(\varepsilon^2+|x|^2)^N}dx\\
&=D_N^{4^*}\varepsilon^N\int_{B_2(0)}\frac{\psi^{4^*}}{(\varepsilon^2+|x|^2)^N}dx+D_N^{4^*}\varepsilon^N\int_{\mathbb{R}^N\backslash B_2(0)}\frac{\psi^{4^*}}{(\varepsilon^2+|x|^2)^N}dx\\
&=D_N^{4^*}\varepsilon^N\int_{B_2(0)}\frac{\psi^{4^*}}{(\varepsilon^2+|x|^2)^N}dx\\
&=D_N^{4^*}\varepsilon^N\omega\int_0^2\frac{r^{N-1}\psi^{4^*}}{(\varepsilon^2+r^2)^{N}}dr\\
&=D_N^{4^*}\varepsilon^N\omega\int_0^2\frac{r^{N-1}}{(\varepsilon^2+r^2)^{N}}dr+D_N^{4^*}\varepsilon^N\omega\int_0^2\frac{r^{N-1}(\psi^{4^*}-1)}{(\varepsilon^2+r^2)^{N}}dr\\
&=D_N^{4^*}\varepsilon^N\omega\int_0^2\frac{r^{N-1}}{(\varepsilon^2+r^2)^{N}}dr+D_N^{4^*}\varepsilon^N\omega\int_1^2\frac{r^{N-1}(\psi^{4^*}-1)}{(\varepsilon^2+r^2)^{N}}dr\\
&=D_N^{4^*}\varepsilon^N\omega\int_0^2\frac{r^{N-1}}{(\varepsilon^2+r^2)^{N}}dr+O(\varepsilon^N)\\
&=D_N^{4^*}\omega\int_0^{\frac{2}{\varepsilon}}\frac{r^{N-1}}{(1+r^2)^N}dr+O(\varepsilon^N).
\end{aligned}
$$
Meanwhile, one also has
$$
\|U_{\varepsilon}\|^{4^*}_{4^*}=D_N^{4^*}\omega\int_0^{\infty}\frac{r^{N-1}}{(1+r^2)^N}dr,
$$
$$
\int_{\frac{2}{\varepsilon}}^{\infty}\frac{r^{N-1}}{(1+r^2)^N}dr\le \int_{\frac{2}{\varepsilon}}^{\infty}r^{-1-N}dr=O(\varepsilon^N)
$$
and
$$
\|\Delta U_\varepsilon\|^2_2=\|U_\varepsilon\|^{4^*}_{4^*}=S^\frac{N}{4}.
$$
Thereby, (\ref{eqn:4*-eatimate}) follows from the above facts evidently.

Finally, we discuss the estimate on $\|u_{\varepsilon}\|^p_p$. Noting that
\begin{equation}\label{eqn:u-p-p-direct-cal-G1-G2}
\begin{aligned}
\|u_{\varepsilon}\|^p_p&=\int_{\mathbb{R}^N}|\psi U_\varepsilon|^pdx\\
&=\int_{\mathbb{R}^N}D_N^p\varepsilon^{\frac{(N-4)p}{2}}\frac{\psi^p}{(\varepsilon^2+|x|^2)^\frac{(N-4)p}{2}}dx\\
&=\int_{B_2(0)}D_N^p\varepsilon^{\frac{(N-4)p}{2}}\frac{\psi^p}{(\varepsilon^2+|x|^2)^\frac{(N-4)p}{2}}dx
+\int_{\mathbb{R}^N\backslash B_2(0)}D_N^p\varepsilon^{\frac{(N-4)p}{2}}\frac{\psi^p}{(\varepsilon^2+|x|^2)^\frac{(N-4)p}{2}}dx\\
&=D_N^p\varepsilon^{\frac{(N-4)p}{2}}\omega\int_0^2\frac{r^{N-1}\psi^p}{(\varepsilon^2+r^2)^{\frac{(N-4)p}{2}}}dr\\
&=D_N^p\varepsilon^{\frac{(N-4)p}{2}}\omega\int_0^2\frac{r^{N-1}}{(\varepsilon^2+r^2)^{\frac{(N-4)p}{2}}}dr
+D_N^p\varepsilon^{\frac{(N-4)p}{2}}\omega\int_0^2\frac{r^{N-1}(\psi^p-1)}{(\varepsilon^2+r^2)^{\frac{(N-4)p}{2}}}dr\\
&=D_N^p\varepsilon^{\frac{(N-4)p}{2}}\omega\int_0^2\frac{r^{N-1}}{(\varepsilon^2+r^2)^{\frac{(N-4)p}{2}}}dr
+D_N^p\varepsilon^{\frac{(N-4)p}{2}}\omega\int_1^2\frac{r^{N-1}(\psi^p-1)}{(\varepsilon^2+r^2)^{\frac{(N-4)p}{2}}}dr\\
&=D_N^p\varepsilon^{N-\frac{(N-4)p}{2}}\omega\int_0^{\frac{2}{\varepsilon}}\frac{r^{N-1}}{(1+r^2)^{\frac{(N-4)p}{2}}}dr+O(\varepsilon^{\frac{(N-4)p}{2}}),
\end{aligned}
\end{equation}
in what follows, we distinguish it into three cases: $p>\frac{N}{N-4}$, $p=\frac{N}{N-4}$ and $p<\frac{N}{N-4}$.

$(i)$ If $p>\frac{N}{N-4}$, it ensures that
$$
\begin{aligned}
&D_N^p\varepsilon^{N-\frac{(N-4)p}{2}}\omega\int_0^{\frac{2}{\varepsilon}}\frac{r^{N-1}}{(1+r^2)^{\frac{(N-4)p}{2}}}dr\\
&=D_N^p\varepsilon^{N-\frac{(N-4)p}{2}}\omega\int_0^{\infty}\frac{r^{N-1}}{(1+r^2)^{\frac{(N-4)p}{2}}}dr-D_N^p\varepsilon^{N-\frac{(N-4)p}{2}}\omega\int_{\frac{2}{\varepsilon}}^{\infty}\frac{r^{N-1}}{(1+r^2)^{\frac{(N-4)p}{2}}}dr\\
&=:K \varepsilon^{N-\frac{(N-4)p}{2}}+o(\varepsilon^{N-\frac{(N-2)p}{2}}).
\end{aligned}
$$
Combining with the fact $\lim\limits_{\varepsilon\rightarrow 0}\frac{\varepsilon^{\frac{(N-4)p}{2}}}{\varepsilon^{N-\frac{(N-4)p}{2}}}=0$, we have
\begin{equation}\label{eqn:0.19}
\|u_{\varepsilon}\|^p_p=K\varepsilon^{N-\frac{(N-4)p}{2}}+o(\varepsilon^{N-\frac{(N-2)p}{2}}).
\end{equation}

$(ii)$ If $p=\frac{N}{N-4}$, it is easy to check that
\begin{equation}\label{eqn:G2-2}
\begin{aligned}
&D_N^p\varepsilon^{\frac{N}{2}}\omega\int_0^{\frac{2}{\varepsilon}}\frac{r^{N-1}}{(1+r^2)^{\frac{N}{2}}}dr\\
&=D_N^p\varepsilon^{\frac{N}{2}}\omega\int_0^1\frac{r^{N-1}}{(1+r^2)^{\frac{N}{2}}}dr
+D_N^p\varepsilon^{\frac{N}{2}}\omega\int_1^{\frac{2}{\varepsilon}}\frac{r^{N-1}}{(1+r^2)^{\frac{N}{2}}}dr\\
&\geq D_N^p\varepsilon^{\frac{N}{2}}\omega\Big(\int_0^1\frac{r^{N-1}}{(1+r^2)^{\frac{N}{2}}}dr+\frac{N}{2^{\frac{N}{2}}}ln\frac{2}{\varepsilon}\Big)\\
&=\left(D_N^p\omega\int_0^1\frac{r^{N-1}}{(1+r^2)^{\frac{N}{2}}}dr+\frac{N}{2^{\frac{N}{2}}}ln 2\right)\varepsilon^{\frac{N}{2}}+\frac{ND_N^p\omega}{2^{\frac{N}{2}}} \varepsilon^{\frac{N}{2}}|ln\varepsilon|.
\end{aligned}
\end{equation}
Taking into account of (\ref{eqn:u-p-p-direct-cal-G1-G2}) and (\ref{eqn:G2-2}), we obtain that
\begin{equation}\label{eqn:0.20}
\|u_{\varepsilon}\|^p_p=:K\varepsilon^{\frac{N}{2}}|ln\varepsilon|+O(\varepsilon^{\frac{N}{2}}).
\end{equation}

$(iii)$ If $2\le p<\frac{N}{N-4}$, we see that
$$
\lim\limits_{\varepsilon\rightarrow 0}\frac{r^{N-1}\psi^p(r)}{(\varepsilon^2+r^2)^{\frac{(N-4)p}{2}}}=\frac{\psi^p(r)}{r^{(N-4)p-(N-1)}} \in L^1([0,2]),
$$
which means that
$$
\begin{aligned}
\|u_{\varepsilon}\|^p_p&=D_N^p\varepsilon^{\frac{(N-4)p}{2}}\omega\int_0^2\frac{r^{N-1}\psi^p(r)}{(\varepsilon^2+r^2)^{\frac{(N-4)p}{2}}}dr\\
&=D_N^p\varepsilon^{\frac{(N-4)p}{2}}\omega\int_0^2\left(\frac{\psi^p(r)}{r^{(N-4)p-(N-1)}}+o(1)\right)dr\\
&=D_N^p\varepsilon^{\frac{(N-4)p}{2}}\omega\int_0^2\frac{\psi^p(r)}{r^{(N-4)p-(N-1)}}dr+o(\varepsilon^{\frac{(N-4)p}{2}})\\
&=:K\varepsilon^{\frac{(N-4)p}{2}}+o(\varepsilon^{\frac{(N-4)p}{2}}).
\end{aligned}
$$
\end{proof}

Thanks to the conclusions in Lemma \ref{lem:bih-eatumate-energy}, the following estimate on $\gamma_\mu(c)$ can be presented.
\begin{lemma}\label{Lem:gamma-mu-estimate}
Under the assumptions of Theorem \ref{Thm:normalized-bScritical-solutions-L2super+L2critical}, we have
\begin{equation}\label{eqn:gamma-mc-estn}
0<\gamma_\mu(c)<\frac{4+\alpha}{2(N+\alpha)}S^\frac{N+\alpha}{4+\alpha}_\alpha.
\end{equation}
\end{lemma}
\begin{proof}
Owing to Lemmas \ref{Lem:gamma-tildegamma-equal} and \ref{Lem:gammamu-infIu}, we only need to show that
$$
m(c)<\frac{4+\alpha}{2(N+\alpha)}S^\frac{N+\alpha}{4+\alpha}_\alpha.
$$
In view of \cite[Theorem 2.1]{Rani-Goyal2022}, we know that
\begin{equation}\label{eqn:S-qlpha-S-ACN}
S_\alpha=\frac{S}{(A_\alpha C(N,\alpha))^\frac{1}{4_\alpha^*}},
\end{equation}
which, together this with \cite[Lemma 4.3]{Chen-chen2023}, (\ref{eqn:H-L-S}) and (\ref{eqn:4*-eatimate}), yields that
\begin{equation}\label{eqn:U4*-leq-estimate}
\begin{aligned}
(A_\alpha C(N,\alpha))^\frac{N}{4}S_\alpha^\frac{N+\alpha}{4}+O(\varepsilon^\frac{N+\alpha}{2})&\leq\int_{\mathbb{R}^N}(I_\alpha*| u_\varepsilon|^{4^*_\alpha})| u_\varepsilon|^{4^*_\alpha}dx\\
&=A_\alpha\int_{\mathbb{R}^N}\int_{\mathbb{R}^N}\frac{|u_\varepsilon(x)|^{4^*_\alpha}|u_\varepsilon(y)|^{4^*_\alpha}}{|x-y|^{N-\alpha}}dxdy\\
&\leq A_\alpha C(N,\alpha) \|u_\varepsilon\|_{4^*}^{24^*_\alpha}\\
&=A_\alpha C(N,\alpha) \Big(S^\frac{N}{4}+O(\varepsilon^{N})\Big)^\frac{N+\alpha}{N}\\
&=A_\alpha C(N,\alpha) \Big[S^\frac{N}{4}_\alpha(A_\alpha C(N,\alpha))^\frac{N(N-4)}{4(N+\alpha)}+O(\varepsilon^{N})\Big]^\frac{N+\alpha}{N}\\
&=(A_\alpha C(N,\alpha))^\frac{N}{4} S^\frac{N+\alpha}{4}_\alpha+O(\varepsilon^{N}).
\end{aligned}
\end{equation}
Define $v_\varepsilon:=\Big(\frac{\|u_\varepsilon\|^2_2}{c}\Big)^\frac{N-4}{8}u_\varepsilon\Big(\Big(\frac{\|u_\varepsilon\|^2_2}{c}\Big)^\frac{1}{4}x\Big)$, then a direct calculation implies that
$$
\int_{\mathbb R^N}|v_\varepsilon|^2dx=c,\ \ \int_{\mathbb R^N}|\Delta v_\varepsilon|^2dx=\int_{\mathbb R^N}|\Delta u_\varepsilon|^2dx
$$
and
$$
\int_{\mathbb R^N}(I_\alpha *|v_\varepsilon|^{4^*_\alpha})|v_\varepsilon|^{4^*_\alpha}dx=\int_{\mathbb R^N}(I_\alpha *|u_\varepsilon|^{4^*_\alpha})|u_\varepsilon|^{4^*_\alpha}dx.
$$
Moreover, for some $K'>0$ and $\varepsilon$ small enough, it follows from (\ref{eqn:p-estimate}) that
\begin{equation}\label{eqn:v-epsilon-estn}
\begin{aligned}
\int_{\mathbb R^N}|v_\varepsilon|^pdx&=\Big(\frac{\|u_\varepsilon\|^2_2}{c}\Big)^\frac{(N-4)p-2N}{8}\int_{\mathbb R^N}|u_\varepsilon|^pdx\\
&\begin{cases} =K' c^\frac{2N-(N-4)p}{8}\varepsilon^{\frac{5}{8}p-\frac{5}{4}}+o(\varepsilon^{\frac{5p}{8}-\frac{5}{4}}), & \text { if } N=5,\frac{18}{5}\leq p<5;\\
 \geq K' c^\frac{2N-(N-4)p}{8}\varepsilon^\frac{15}{8}|ln\varepsilon|+o\Big(\varepsilon^\frac{15}{8}|ln \varepsilon|\Big), & \text { if } N=5,p = 5; \\
= K' c^\frac{2N-(N-4)p}{8}\varepsilon^{\frac{3}{4}(5-\frac{p}{2})}+o(\varepsilon^{\frac{15}{4}-\frac{3p}{8}}), & \text { if } N=5,p>5; \\
 =K' c^\frac{2N-(N-4)p}{8}\varepsilon^{\frac{8-N}{4}(N-\frac{(N-4)p}{2})}+o(\varepsilon^{\frac{8-N}{4}(N-\frac{(N-4)p}{2})}), & \text { if } N=6,7; \\
 \geq K' c^\frac{2N-(N-4)p}{8}|ln \varepsilon|^\frac{p-4}{2}+o(1) |ln \varepsilon|^\frac{p-4}{2}+O(1), & \text { if } N=8; \\
 =K' c^\frac{2N-(N-4)p}{8}+o(1), & \text { if } N\geq 9.\\
\end{cases}
\end{aligned}
\end{equation}

Using Proposition \ref{Prop:I-u-maximum}, we find that there exists $s_{v_\varepsilon}$ such that $\mathcal H(v_\varepsilon,s_{v_\varepsilon})\in \mathcal P_r(c)$. Combining this with (\ref{eqn:delta-eatimate}), (\ref{eqn:p-estimate}), (\ref{eqn:U4*-leq-estimate}) and (\ref{eqn:v-epsilon-estn}), we conclude that
there holds that
\begin{equation}\label{eqn:mc-estimate}
\begin{aligned}
m(c)&\leq I(\mathcal H(v_\varepsilon,s_{v_\varepsilon}))\\
&=\frac{e^{4s_{v_\varepsilon}}}{2}\int_{\mathbb{R}^N}|\Delta v_\varepsilon|^2dx-\frac{\mu e^{2p\gamma_p s_{v_\varepsilon}}}{p}\int_{\mathbb{R}^N}|v_\varepsilon|^p dx-\frac{e^{44^*_\alpha s_{v_\varepsilon}}}{24^*_\alpha}\int_{\mathbb{R}^N}(I_\alpha*|v_\varepsilon|^{4^*_\alpha})|v_\varepsilon|^{4^*_\alpha} dx\\
&\leq \frac{e^{4s_{v_\varepsilon}}}{2}\Big(S^\frac{N}{4}+O(\varepsilon^{N-4})\Big)-\frac{e^{44^*_\alpha s_{v_\varepsilon}}}{24^*_\alpha}\Big((A_\alpha C(N,\alpha))^\frac{N}{4}S_\alpha^\frac{N+\alpha}{4}+O(\varepsilon^\frac{N+\alpha}{2})\Big)\\
&\ \ \ -\frac{\mu e^{2p\gamma_p s_{v_\varepsilon}}}{p}K' c^\frac{2N-(N-4)p}{8}
\begin{cases}
\varepsilon^{\frac{5}{8}p-\frac{5}{4}}+o(\varepsilon^{\frac{5}{8}p-\frac{5}{4}}), & \mbox{if} \,\, N=5,\frac{18}{5}\leq p<5;\\
\varepsilon^\frac{15}{8}|ln\varepsilon|+o(\varepsilon^\frac{15}{8}|ln\varepsilon|), & \mbox{if} \,\,p = 5; \\
\varepsilon^{\frac{3}{4}(N-\frac{(N-4)p}{2})}+o(\varepsilon^{\frac{3}{4}(N-\frac{(N-4)p}{2})}), & \mbox{if} \,\, N=5,p>5; \\
\varepsilon^{\frac{8-N}{4}(N-\frac{(N-4)p}{2})}+o(\varepsilon^{\frac{8-N}{4}(N-\frac{(N-4)p}{2})}), &  \mbox{if} \,\, N=6,7; \\
|ln \varepsilon|^\frac{p-4}{2}+o(1) |ln \varepsilon|^\frac{p-4}{2}+O(1), &  \mbox{if} \,\, N=8; \\
1+o(1), &  \mbox{if} \,\, N\geq 9\\
\end{cases}\\
&=\frac{e^{4s_{v_\varepsilon}}}{2}\Big((A_\alpha C(N,\alpha))^\frac{N(N-4)}{4(N+\alpha)}S^\frac{N}{4}_\alpha+O(\varepsilon^{N-4})\Big)-\frac{e^{44^*_\alpha s_{v_\varepsilon}}}{24^*_\alpha}\Big((A_\alpha C(N,\alpha))^\frac{N}{4}S_\alpha^\frac{N+\alpha}{4}+O(\varepsilon^\frac{N+\alpha}{2})\Big)\\
&\ \ \ -\frac{\mu e^{2p\gamma_p s_{v_\varepsilon}}}{p}K' c^\frac{2N-(N-4)p}{8}
\begin{cases}
\varepsilon^{\frac{5}{8}p-\frac{5}{4}}+o(\varepsilon^{\frac{5}{8}p-\frac{5}{4}}), &  \mbox{if} \,\, N=5,\frac{18}{5}\leq p<5;\\
\varepsilon^\frac{15}{8}|ln\varepsilon|+o(\varepsilon^\frac{15}{8}|ln\varepsilon|), &  \mbox{if} \,\, N=5,p = 5; \\
\varepsilon^{\frac{3}{4}(N-\frac{(N-4)p}{2})}+o(\varepsilon^{\frac{3}{4}(N-\frac{(N-4)p}{2})}), &  \mbox{if} \,\, N=5,p>5; \\
\varepsilon^{\frac{8-N}{4}(N-\frac{(N-4)p}{2})}+o(\varepsilon^{\frac{8-N}{4}(N-\frac{(N-4)p}{2})}), &  \mbox{if} \,\, N=6,7; \\
|ln \varepsilon|^\frac{p-4}{2}+o(1) |ln \varepsilon|^\frac{p-4}{2}+O(1), &  \mbox{if} \,\, N=8; \\
1+o(1), &  \mbox{if} \,\, N\geq 9.\\
\end{cases}
\end{aligned}
\end{equation}

We claim that there exist $e^{s_0}, e^{s_1}>0$ independent of $\varepsilon$ such that $e^{s_{v_\varepsilon}} \in[e^{s_0}, e^{s_1}]$ for $\varepsilon>0$ small enough. Suppose by contradiction that $e^{s_{v_\varepsilon}} \rightarrow +\infty$ or $e^{s_{v_\varepsilon}} \rightarrow 0$ as $\varepsilon \rightarrow 0$, that is, $s_{v_\varepsilon} \rightarrow +\infty$ or $s_{v_\varepsilon} \rightarrow -\infty$ as $\varepsilon \rightarrow 0$. The fact that $\mathcal H(v_\varepsilon,s_{v_\varepsilon})\in \mathcal P_r(c)$ yields that
\begin{equation}\label{eqn:contradiction-prove}
\int_{\mathbb{R}^N}|\Delta v_\varepsilon|^2dx=\mu\gamma_pe^{(2p\gamma_p-4) s_{v_\varepsilon}}\int_{\mathbb{R}^N}|v_\varepsilon|^p dx+e^{(44^*_\alpha-4) s_{v_\varepsilon}}\int_{\mathbb{R}^N}(I_\alpha*|v_\varepsilon|^{4^*_\alpha})|v_\varepsilon|^{4^*_\alpha} dx.
\end{equation}
However, when $p>2+\frac{8}{N}$, the left side of (\ref{eqn:contradiction-prove}) is $S^{\frac{N}{4}}+O(\varepsilon^{N-4})>0$ for small $\varepsilon$, while its right side tends to $\infty$ as $s_{v_\varepsilon} \rightarrow +\infty$ and tends to $0$ as $s_{v_\varepsilon} \rightarrow -\infty$, which is impossible; when $p=\bar{p}$ and $s_{v_\varepsilon} \rightarrow +\infty$, we find a similar contradiction since the right side of (\ref{eqn:contradiction-prove}) tends to $\infty$; when $p=\bar{p}$, it follows from $c\in(0,c_0)$ and (\ref{eqn:GNinequality}) that
$$
\mu\gamma_{\bar{p}}\|v_\varepsilon\|^{\bar{p}}_{\bar{p}}\leq \mu\gamma_{\bar{p}}C^{\bar{p}}_{N,{\bar{p}}}c^\frac{4}{N}\|\Delta v_\varepsilon\|^2_2<\mu\gamma_{\bar{p}}C^{\bar{p}}_{N,{\bar{p}}} \frac{\bar{p}}{4\mu C^{\bar{p}}_{N,{\bar{p}}}}\|\Delta v_\varepsilon\|^2_2<\|\Delta v_\varepsilon\|^2_2=S^{\frac{N}{4}}+O(\varepsilon^{N-4}),
$$
which means that (\ref{eqn:contradiction-prove}) is absurd when $s_{v_\varepsilon} \rightarrow -\infty$. Thus, the claim holds.

Observe that
\begin{itemize}
\item[$(i)$] if $N=5$, $p>\frac{22}{3}$, it leads to $\frac{3}{4}(5-\frac{(5-4)p}{2})<1<\frac{5+\alpha}{2}$;
\item[$(ii)$] if $N = 6,7,8$, it gives that $N-4<\frac{N+\alpha}{2}$.
\end{itemize}
Then, $O(\varepsilon^{N-4})$ and $O(\varepsilon^\frac{N+\alpha}{2})$ in (\ref{eqn:mc-estimate}) can be controlled by the last term for $\varepsilon>0$ small enough. Consequently, we infer that
\begin{equation}\label{eqn:mc-maximum}
\begin{aligned}
m(c) &\leq I(\mathcal H(v_\varepsilon,s_{v_\varepsilon}))\\
&<\sup _{s > 0}\Big(\frac{e^{4s}}{2}(A_\alpha C(N,\alpha))^\frac{N(N-4)}{4(N+\alpha)}S^\frac{N}{4}_\alpha-\frac{e^{44^*_\alpha s}}{24^*_\alpha}(A_\alpha C(N,\alpha))^\frac{N}{4}S_\alpha^\frac{N+\alpha}{4}\Big).
\end{aligned}
\end{equation}
To reach the conclusion, consider
$$
f(s):=\frac{e^{4s}}{2}(A_\alpha C(N,\alpha))^\frac{N(N-4)}{4(N+\alpha)}S^\frac{N}{4}_\alpha-\frac{e^{44^*_\alpha s}}{24^*_\alpha}(A_\alpha C(N,\alpha))^\frac{N}{4}S_\alpha^\frac{N+\alpha}{4}.
$$
A direct calculation gives that
$$
f'(s)=2e^{4s}(A_\alpha C(N,\alpha))^\frac{N(N-4)}{4(N+\alpha)}S^\frac{N}{4}_\alpha-2e^{44^*_\alpha s}(A_\alpha C(N,\alpha))^\frac{N}{4}S_\alpha^\frac{N+\alpha}{4},
$$
which means $f(s)$ has a unique maximum point $s_0$ satisfying
$$
\begin{aligned}
e^{ s_0}&=(A_\alpha C(N,\alpha))^{-\frac{N(4+\alpha)}{4(N+\alpha)}\frac{1}{(44^*_\alpha-4)}}S^{-\frac{\alpha}{4}\frac{1}{(44^*_\alpha-4)}}_\alpha\\
&=(A_\alpha C(N,\alpha))^{-\frac{N(4+\alpha)}{4(N+\alpha)}\frac{N-4}{4(4+\alpha)}}S^{-\frac{\alpha}{4}\frac{N-4}{4(4+\alpha)}}_\alpha\\
&=(A_\alpha C(N,\alpha))^{-\frac{N(N-4)}{16(N+\alpha)}}S^{-\frac{\alpha(N-4)}{16(4+\alpha)}}_\alpha.
\end{aligned}
$$
Therefore, we have
$$
\begin{aligned}
f(s_0)&=\frac{e^{4s_0}}{2}(A_\alpha C(N,\alpha))^\frac{N(N-4)}{4(N+\alpha)}S^\frac{N}{4}_\alpha-\frac{e^{44^*_\alpha s_0}}{24^*_\alpha}(A_\alpha C(N,\alpha))^\frac{N}{4}S_\alpha^\frac{N+\alpha}{4}\\
&=\frac{1}{2}(A_\alpha C(N,\alpha))^{-\frac{N(N-4)}{4(N+\alpha)}}S^{-\frac{\alpha(N-4)}{4(4+\alpha)}}_\alpha(A_\alpha C(N,\alpha))^\frac{N(N-4)}{4(N+\alpha)}S^\frac{N}{4}_\alpha\\
&\ \ \ -\frac{1}{24^*_\alpha}(A_\alpha C(N,\alpha))^{-\frac{N(N-4)}{16(N+\alpha)}\frac{4(N+\alpha)}{N-4}}S^{-\frac{\alpha(N-4)}{16(4+\alpha)}\frac{4(N+\alpha)}{N-4}}_\alpha
(A_\alpha C(N,\alpha))^\frac{N}{4}S_\alpha^\frac{N+\alpha}{4}\\
&=\frac{1}{2}S^{-\frac{\alpha(N-4)}{4(4+\alpha)}}_\alpha S^\frac{N}{4}_\alpha-\frac{1}{24^*_\alpha}S_\alpha^{-\frac{\alpha(N-\alpha)}{4(4+\alpha)}}S_\alpha^\frac{N+\alpha}{4}\\
&=\Big(\frac{1}{2}-\frac{1}{24^*_\alpha}\Big)S^{\frac{N+\alpha}{4+\alpha}}_\alpha\\
&=\frac{4+\alpha}{2(N+\alpha)}S_\alpha^\frac{N+\alpha}{4+\alpha}.
\end{aligned}
$$
Together this with (\ref{eqn:mc-maximum}), Lemmas \ref{Lem:gamma-tildegamma-equal} and \ref{Lem:gammamu-infIu}, we conclude that
$$
0<m(c)=\gamma_\mu(c)<\frac{4+\alpha}{2(N+\alpha)}S_\alpha^\frac{N+\alpha}{4+\alpha}.
$$
\end{proof}

\textbf{Proof of Theorem \ref{Thm:normalized-bScritical-solutions-L2super+L2critical}.} Lemma \ref{Lem:PS-true} and Lemma \ref{Lem:gamma-mu-estimate} ensure the existence of a \emph{(PS) }sequence $\{u_n\}\subset S_r(c)$ for $I|_{S_r(c)}$ at level $\gamma_\mu(c)$ satisfying (\ref{eqn:gamma-mc-estn}) and $P(u_n)\to 0$ as $n\to\infty$. Thus, one of the two alternatives in Lemma \ref{pro:solution-either-or-true} holds. Suppose that $(i)$ of Lemma \ref{pro:solution-either-or-true} occurs, that is, up to a subsequence, there exists $u_{\mu,c} \in H^2_r(\mathbb{R}^N)$ such that $u_n \rightharpoonup u_{\mu,c} \neq 0$ in $H^2_r(\mathbb{R}^N)$ and $u_{\mu,c}$ solves (\ref{eqn:BS-equation-L2-Super+Critical}) for some $\lambda_{\mu,c}<0$. Meanwhile, we also have
\begin{equation}\label{eqn:Iu-geq-0}
I(u_{\mu,c}) \leq \gamma_\mu(c)-\frac{4+\alpha}{2(N+\alpha)}S_\alpha^\frac{N+\alpha}{4+\alpha}<0.
\end{equation}
However, since $P(u_{\mu,c})=0$ and $u_{\mu,c} \neq 0$, for $\bar{p}\leq p<2_\alpha^*$, it gives that
$$
I(u_{\mu,c})=\Big(\frac{\gamma_p}{2}-\frac{1}{p}\Big)\mu \|u_{\mu,c}\|^p_p+\Big(\frac{1}{2}-\frac{1}{2 4_\alpha^*}\Big) \int_{\mathbb R^N}(I_\alpha*|u_{\mu,c}|^{4^*_\alpha})|u_{\mu,c}|^{4^*_\alpha}dx>0,
$$
an obvious contradiction with (\ref{eqn:Iu-geq-0}). Then, $(ii)$ of Lemma \ref{pro:solution-either-or-true} holds true, in other words, $u_n \rightarrow u_{\mu,c}$ in $H^2_r(\mathbb{R}^N)$ and $u_{\mu,c}$ is a normalized solution to (\ref{eqn:BS-equation-L2-Super+Critical}) for $\lambda_{\mu,c}<0$ with $I(u_{\mu,c})=\gamma_\mu(c)$.

In what follows, we demonstrate that $u_{\mu,c}$ is nonnegative. Indeed, since the functional $I(u)$ is even, the continuous path $g_n(t)=((g_n)_1(t),0)$ used in Lemma \ref{Lem:PS-true}-\emph{(iv)} to verify the inequality
$$
\min\limits_{t\in[0,1]}\|(v_n,s_n)-g_n(t)\|^2_E\leq\frac{1}{n}\ \mbox{for all}\ n\in \mathbb N
$$
can be fixed with $(g_n)_1(t)\geq0$ for all $t\in[0,1]$. From this fact, it follows that there exists $t_n\in[0,1]$ such that $\nu_n:=(g_n)_1(t_n)\geq0$ and
$$
\|v_n-\nu_n\|^2_{H^2}+|s_n|^2_{\mathbb{R}}=\min\limits_{t\in[0,1]}\|(v_n,s_n)-g_n(t)\|^2_E\leq\frac{1}{n}\ \mbox{for all}\ n\in \mathbb N.
$$
As a consequence, we infer that
$$
\|\Delta \mathcal H(v_n,s_n)-\Delta \mathcal H(\nu_n,s_n)\|^2_2=e^{4s_n}\|\Delta v_n-\Delta \nu_n\|^2_2\leq\frac{e^{4s_n}}{n}
$$
and
$$
\|\mathcal H(v_n,s_n)-\mathcal H(\nu_n,s_n)\|^2_2=\|v_n-\nu_n\|^2_2\leq\frac{1}{n},
$$
which state that $\|u_n-\mathcal H(\nu_n,s_n)\|_{H^2}\!\to \!0$ as $n\!\to\! \infty.$ As
$$
\mathcal H(\nu_n,s_n)(x)=e^\frac{Ns_n}{2}\nu_n(e^{s_n}x)\geq0
$$ and $u_n\to u_{\mu,c}$ in
$H^2_r(\mathbb{R}^N)$, up to a subsequence, we must have $u_{\mu,c}\geq0$ for almost every $x\in\mathbb{R}^N$.

Lastly, we discuss the asymptotic behavior in Theorem in \ref{Thm:normalized-bScritical-solutions-L2super+L2critical}. Note that $u_{\mu,c}$ belongs to $\mathcal P_r(c)$, that is, $P(u_{\mu,c})=0$. Then, it brings that
$$
\begin{aligned}
\frac{4+\alpha}{2(N+\alpha)}S_\alpha^\frac{N+\alpha}{4+\alpha}&>I(u_{\mu,c})\\
&=I(u_{\mu,c})-\frac{1}{24^*_\alpha}P(u_{\mu,c})\\
&= \Big(\frac{1}{2}-\frac{1}{24^*_\alpha}\Big)\int_{\mathbb{R}^N}|\Delta u_{\mu,c}|^2 dx+\Big(\frac{\gamma_p}{24^*_\alpha}-\frac{1}{p}\Big)\mu\int_{\mathbb R^N}|u_{\mu,c}|^pdx\\
&\geq \Big(\frac{1}{2}-\frac{1}{24^*_\alpha}\Big)\int_{\mathbb{R}^N}|\Delta u_{\mu,c}|^2 dx,
\end{aligned}
$$
which means that $\{u_{\mu,c}\}$ is uniformly bounded in $H^2(\mathbb{R}^N)$ with respect to $\mu>0$. Therefore, the continuity of $H^2(\mathbb R^N)\hookrightarrow L^s(\mathbb R^N)$ for $s\in[2,4^*]$ indicates that
\begin{equation}\label{eqn:lim-p-0}
\lim _{\mu \to 0^+} \mu \int_{\mathbb R^N}|u_{\mu,c}|^pdx=0.
\end{equation}
Once again, employing $P(u_{\mu,c})=0$ ensures that
\begin{equation}\label{eqn:nabla-2-star=0}
\lim _{\mu \to 0^+}\|\Delta u_{\mu,c}\|_2^2
=\lim _{\mu \rightarrow 0^+} \int_{\mathbb R^N}(I_\alpha*|u_{\mu,c}|^{4^*_\alpha})|u_{\mu,c}|^{4^*_\alpha}dx:=l,
\end{equation}
up to a subsequence. According to (\ref{eqn:S-alpha-define}), it can be inferred that either
$$
l=0 \quad \text { or } \quad l \geq S_\alpha^{\frac{N+\alpha}{4+\alpha}}.
$$
If $l=0$, obviously (\ref{eqn:lim-p-0}) and (\ref{eqn:nabla-2-star=0}) signify that $I(u_{\mu,c}) \to 0$ as $\mu \to 0^+$, which
contradicts with the fact obtained in (\ref{eqn:gamma-mu-widegamma-mu}). At this stage, taking into account $P(u_{\mu,c})=0$ and (\ref{eqn:lim-p-0}), we
derive that
$$
\begin{aligned}
\frac{4+\alpha}{2(N+\alpha)} S_\alpha^{\frac{N+\alpha}{4+\alpha}}&\!>\!I(u_{\mu,c}) \\
& =\frac{4+\alpha}{2(N+\alpha)}\|\Delta u_{\mu,c}\|_2^2+o_\mu(1) \\
& \geq \frac{4+\alpha}{2(N+\alpha)} S_\alpha^{\frac{N+\alpha}{4+\alpha}}+o_\mu(1),
\end{aligned}
$$
which yields that
$$
\lim _{\mu \to 0^+} I_{\mu,\gamma}(u_{\mu,c})=\frac{4+\alpha}{2(N+\alpha)} S_\alpha^{\frac{N+\alpha}{4+\alpha}}.
$$

To discuss asymptotical behavior as $c\rightarrow 0^+$, let $(u_{\mu,c}, \lambda_{u_{\mu,c}})$ be a solution of problem (\ref{eqn:BS-equation-L2-Super+Critical}) and set
$$
v_{\mu,c}(x):=(-\lambda_{u_{\mu,c}})^{-\frac{N-4}{8}}{u_{\mu,c}}((-\lambda_{u_{\mu,c}})^{-\frac{1}{4}} x).
$$
Then, it is easy to verify that $(v_{\mu,c}, t_{\mu,c})$ solves the following problem
\begin{equation}\label{eqn:v-u-equation}
\left\{\begin{array}{l}
{\Delta}^2v-(I_\alpha*|v|^{4^*_\alpha})|v|^{4^*_\alpha-2}v-t_{\mu,c}|v|^{p-2}v= v\ \ \mbox{in}\ \mathbb{R}^N, \\[0.1cm]
\int_{\mathbb{R}^N}v^2 d x=c \left(\frac{t_{\mu,c}}{\mu }\right)^\frac{8}{(N-4)p-2N},
\end{array}\right.
\end{equation}
where $t_{\mu,c}:=\mu (-\lambda_{u_{\mu,c}})^\frac{(N-4)p-2N}{8}$. Conversely, assume that $(v_{\mu,c}, t_{\mu,c})$ is a solution of
problem (\ref{eqn:v-u-equation}), denote by
$$
\lambda_{u_{\mu,c}}:=-\left(\frac{t_{\mu,c}}{\mu }\right)^\frac{8}{(N-4)p-2N}\ \ \ \mbox{and}\ \ \
u_{\mu,c}(x):=(-\lambda_{u_{\mu,c}})^\frac{N-4}{8} v_{\mu,c}((-\lambda_{u_{\mu,c}})^\frac{1}{4} x),
$$
it is obvious that $(u_{\mu,c},\lambda_{u_{\mu,c}})$ is a solution of problem (\ref{eqn:BS-equation-L2-Super+Critical}). Furthermore,
with regard to solution $(v_{\mu,c}, t_{\mu,c})$ of problem (\ref{eqn:v-u-equation}), considering
$$
\eta_{\mu,c}:=-t_{\mu,c}^\frac{8}{(N-4)p-2N}\ \ \ \mbox{and}\ \ \  \omega_{\mu,c}(x):=(-\eta_{\mu,c})^\frac{N-4}{8} v_{\mu,c}((-\eta_{\mu,c})^\frac{1}{4}x),
$$
we see that $(\omega_{\mu,c},\eta_{\mu,c})$ is a solution for the following problem
\begin{equation}\label{eqn:omega-eta}
\left\{\begin{array}{l}
{\Delta}^2\omega-(I_\alpha*|\omega|^{4^*_\alpha})|\omega|^{4^*_\alpha-2}v-|u|^{p-2}u=\eta \omega\ \ \mbox{in}\ \mathbb{R}^N, \\[0.1cm]
\int_{\mathbb{R}^N}\omega^2 d x=c \mu^\frac{8}{2N-(N-4)p}.
\end{array}\right.
\end{equation}
Meanwhile, if $(\omega_{\mu,c}, \eta_{\mu,c})$ is a solution of problem (\ref{eqn:omega-eta}), then by letting
$$
t_{\mu,c}:=(-\eta_{\mu,c})^\frac{(N-4)p-2N}{8}\ \ \ \mbox{and}\ \ \ v_{\mu,c}(x):=(-\eta_{\mu,c})^{-\frac{N-4}{8}} \omega_{\mu,c}((-\eta_{\mu,c})^{-\frac{1}{4}} x),
$$
we readily check that $(v_{\mu,c},t_{\mu,c})$ is also a solution of (\ref{eqn:v-u-equation}).

Based on the above arguments, finding solutions of problem (\ref{eqn:BS-equation-L2-Super+Critical}) is equivalent to investigate the existence of solutions of (\ref{eqn:omega-eta}). In particular, the asymptotical behavior as $c\to 0^+$ is same to what happened when $\mu\to 0^+$, that is, (\ref{eqn:lim-mc-c0}) holds true.
\qed
\\

\noindent\textbf{Acknowledgements}
Z.H. Zhang was partly supported by the NSFC (Grant No. 12371402). H.R. Sun was partly supported by the NSF of Gansu Province of China (Grant No. 21JR7RA535, 24JRRA414).
\\[-0.1cm]

\noindent \textbf{Data Availability}
The authors declare that data sharing is not applicable to this article as no data sets were generated or analyzed during the current study.
\\[-0.1cm]

\noindent\textbf{Conflict of interest} The authors declare that they have no conflict of interest.

\bibliographystyle{elsarticle-num}

\end{document}